\documentclass[12pt,leqno]{amsart}

\usepackage[english]{babel}
\usepackage[utf8]{inputenc}

\usepackage[oldstylemath,fulloldstylenums]{kpfonts}

\usepackage{fullpage}

\usepackage{graphicx}
\usepackage{indentfirst}

\usepackage{enumerate,enumitem}

\usepackage{setspace}
\usepackage{amsmath,amsthm,mathtools,mathrsfs,amssymb} 

\usepackage{thmtools} % for correct autorefs

\usepackage{xspace} % to get the spacing after macros right

\usepackage{xcolor}
\definecolor{acid}{HTML}{479107}
\definecolor{crimson}{RGB}{177,12,12}
\definecolor{teal}{HTML}{078c91}
\definecolor{orange}{HTML}{ef9e8b}

\usepackage{hyperref}
\hypersetup{
    colorlinks=true,
    linktocpage=true, 
    breaklinks=true, 
    pageanchor=true,
    pdfpagemode=UseNone,
    urlcolor=orange, 
    linkcolor=teal, 
    citecolor=acid,
    pdfcreator={pdfLaTeX},
    pdfproducer={LaTeX with hyperref}
}

\usepackage{braket} % for nicer < > 

\usepackage{calc} % to allow adding lengths

\usepackage{pgf}
\usepackage{tikz-cd}
\usetikzlibrary{math,arrows,positioning,shapes,fit,calc,positioning}
\usepackage{quiver}
\usepackage{multicol}

\newtheorem{theorem}{Theorem}[section]
\newtheorem{lemma}[theorem]{Lemma}

\newtheorem{corollary}[theorem]{Corollary}
\newtheorem{proposition}[theorem]{Proposition}

\theoremstyle{definition}
\newtheorem{definition}[theorem]{Definition}

\newtheorem*{question}{Question}
\newtheorem{example}[theorem]{Example}
\newtheorem*{example*}{Example}

\theoremstyle{remark}
\newtheorem{remark}[theorem]{Remark}

\newenvironment{itemize*}
  {\begin{itemize}[label={--}]}
  {\end{itemize}}
  
\newenvironment{enumerate*}
  {\begin{enumerate}[label=(\alph*),topsep=-\parskip+\jot,itemsep=-\parskip-\jot]}
  {\end{enumerate}}
  
\newenvironment{enumerate**}
  {\begin{enumerate}[label=\roman*.]}
  {\end{enumerate}}
  
\newenvironment{enumerate***}
  {\begin{enumerate}[label=(\alph*'),topsep=-\parskip+\jot,itemsep=-\parskip-\jot]}
  {\end{enumerate}}

\newcommand{\Hom}[3]{\operatorname{Hom}_{#1}({#2}, {#3})}
\let\Hom\varHom

\newcommand{\iso}{\simeq} % one line 
\renewcommand{\equiv}{\simeq}

\newcommand{\inv}[1]{{#1}^{\kern0.1em \text{-}1}}

\DeclareMathOperator{\id}{id}

\newcommand{\epi}{\twoheadrightarrow}

\newcommand{\inc}{\hookrightarrow}

\newcommand{\op}[1]{{#1}^{\operatorname{op}}}

\DeclareMathOperator{\C}{\mathsf{C}}

\DeclareMathOperator{\Frm}{\mathsf{Frm}}

\DeclareMathOperator{\Preord}{\mathsf{Preord}}

\newcommand{\Trip}[1]{\mathsf{Trip}({#1})} 
\newcommand{\Tripcart}[1]{\mathsf{Trip}_{\mathsf{cart}}({#1})} 
\newcommand{\Tripgeom}[1]{\mathsf{Trip}_{\mathsf{geom}}({#1})} 
\newcommand{\SubTrip}[1]{\mathsf{SubTrip}({#1})} 
\newcommand{\Tripreg}[1]{\mathsf{Trip}_{\mathsf{reg}}({#1})} 
\newcommand{\ClTrans}[1]{\mathsf{ClTrans}({#1})} 
\newcommand{\SubTop}[1]{\mathsf{SubTop}({#1})} 

\newcommand{\sh}[1]{\operatorname{Sh}(#1)}

\newcommand{\bigland}{\bigwedge}
\newcommand{\biglor}{\bigvee}
\renewcommand{\implies}{\Rightarrow}

\newbox\gnBoxA
\newdimen\gnCornerHgt
\setbox\gnBoxA=\hbox{$\ulcorner$}
\global\gnCornerHgt=\ht\gnBoxA
\newdimen\gnArgHgt
\def\name #1{%
	\setbox\gnBoxA=\hbox{$#1$}%
	\gnArgHgt=\ht\gnBoxA%
	\ifnum     \gnArgHgt<\gnCornerHgt \gnArgHgt=0pt%
	\else \advance \gnArgHgt by -\gnCornerHgt%
	\fi \raise\gnArgHgt\hbox{$\ulcorner$} \box\gnBoxA %
	\raise\gnArgHgt\hbox{$\urcorner$}}

\makeatletter
\providecommand*{\dashv}{%
	\mathrel{%
		\mathpalette\@dashv\vdash
	}%
}
\newcommand*{\@dashv}[2]{%
	\reflectbox{$\m@th#1#2$}%
}
\makeatother

\renewcommand{\bar}[1]{\overline{#1}}

\renewcommand{\tilde}[1]{\widetilde{#1}}
\let\bf\mathbf

\DeclareMathOperator{\A}{\mathcal{A}}
\DeclareMathOperator{\B}{\mathcal{B}}

\let\AA\relax
\DeclareMathOperator{\AA}{\mathbb{A}}
\DeclareMathOperator{\BB}{\mathbb{B}}

\DeclareMathOperator{\Sets}{\mathsf{Set}}
\DeclareMathOperator{\E}{\mathsf{E}}
\newcommand{\Eff}{\mathsf{Eff}}
\newcommand{\Mod}{\mathsf{Mod}}
\newcommand{\RT}[1]{\mathsf{RT}({#1})}
\DeclareMathOperator{\oPCA}{\mathsf{OPCA}}
\DeclareMathOperator{\ArrAlg}{\mathsf{ArrAlg}}
\DeclareMathOperator{\ArrAlgcd}{\mathsf{ArrAlg}_{\mathsf{cd}}}
\DeclareMathOperator{\ArrAlgreg}{\mathsf{ArrAlg}_{\mathsf{reg}}}
\DeclareMathOperator{\ArrAlgbi}{\mathsf{ArrAlg}_{\mathsf{bi}}}

\DeclareMathOperator{\ArrAlgmod}{\mathsf{ArrAlg}_{\mathsf{mod}}}
\DeclareMathOperator{\HeytPre}{\mathsf{HeytPre}}

\newcommand{\AT}[1]{\mathsf{AT}({#1})}
\newcommand{\N}[1]{\mathsf{N}({#1})}
\newcommand{\Pow}[1]{\operatorname{Pow}({#1})}

\let\O\relax
\DeclareMathOperator{\O}{\mathcal{O}}

\newcommand{\downl}[1]{{\downarrow}{#1}}

\newcommand{\downr}[1]{{#1}{\downarrow}}

\newcommand{\s}[1]{{#1}^{\arr}}

\newcommand{\arr}{\rightarrow}

\let\to\relax
\newcommand{\to}{\longrightarrow}

\newcommand{\evleq}{\preccurlyeq}

\newcommand{\evm}{\curlywedge}
\makeatletter

\DeclareRobustCommand\bigop[1]{%
	\mathop{\vphantom{\sum}\mathpalette\bigop@{#1}}\slimits@
}
\newcommand{\bigop@}[2]{%
	\vcenter{%
		\sbox\z@{$#1\sum$}%
		\hbox{\resizebox{\ifx#1\displaystyle0.85\fi\dimexpr\ht\z@+\dp\z@}{!}{$\m@th#2$}}%
	}%
}
\makeatother

\newcommand{\evmeet}{\DOTSB\bigop{\evm}}

\newcommand{\liso}{\dashv\vdash}

\newcommand{\numberset}[1]{{\mathbb{#1}}}
\newcommand{\NN}{\numberset{N}}

\newcommand{\myTitle}{\MakeUppercase{A category of arrow algebras for modified realizability}\xspace}
\newcommand{\myName}{Umberto Tarantino\xspace}
\newcommand{\myDepartment}{IRIF\xspace}
\newcommand{\myUni}{Université Paris Cité\xspace}

\newcommand{\Acknowledgements}{{
    \bigskip
    \footnotesize
    
    {\textbf {Acknowledgements.}}
    The author expresses his gratitude towards Benno van den Berg for supervising the above work as part of his Master Thesis.
}}

\begin{document}

\title{\myTitle} 
\author{\myName}
\date{\today}
\address{\myDepartment, \myUni}
\email{tarantino@irif.fr}

\begin{abstract}
In this paper we further the study of arrow algebras, simple algebraic structures inducing toposes through the tripos-to-topos construction, by defining appropriate notions of morphisms between them which correspond to morphisms of the associated triposes. Specializing to geometric inclusions, we characterize subtriposes of an arrow tripos in terms of nuclei on the underlying arrow algebra, recovering a classical locale-theoretic result. As an example of application, we lift modified realizability to the setting of arrow algebras, and we establish its functoriality.
\end{abstract} 

\maketitle

% \let\thefootnote\relax
% \footnotetext{To appear in } 

\bigskip
The purpose of this paper is to develop the theory of arrow algebras as a framework to study realizability toposes from a more concrete, `algebraic', point of view which can also take localic toposes into account. 

Arrow algebras were introduced in \cite{berg2023arrow}, of which this paper can be seen as a follow-up, generalizing Alexandre Miquel's \emph{implicative algebras} \cite{miquel2020-1, miquel2020-2} as algebraic structures which induce triposes and then toposes through the tripos-to-topos construction. The weakening of the axiom that distinguishes implicative algebras from arrow algebras allows the latter to perfectly factor through the construction of realizability triposes coming from partial combinatory algebras -- from now on, PCAs -- which are actually \emph{partial}, whereas implicative algebras are the intermediate structure only in the case of \emph{total} combinatory algebras. Indeed, in \cite{berg2023arrow} it is shown how every frame can be seen as an arrow algebra in such a way that the induced \emph{arrow tripos} coincides with the usual localic tripos; similarly, every PCA gives rise to an arrow algebra in such a way that the induced arrow tripos coincides with the usual realizability tripos. The aim of the following is then to define appropriate notions of morphisms between arrow algebras which correspond to morphisms of the associated triposes, so as to determine a category of arrow algebras factoring through the construction of both realizability and localic triposes in a 2-functorial way. 

\subsection*{Related work}
Earlier work in this area includes Hofstra's \emph{basic combinatory objects} (BCOs) \cite{hofstra2006}, fundamental in advocating for the study of \emph{relative} realizability which is implemented by the theory of both implicative and arrow algebras through the \emph{separator}. The notion of a BCO appears quite different and more combinatorial in nature than the more algebraic notion of an arrow algebra, and we still don't see how exactly the two structures relate: we hope to investigate this further in the future, possibly showing how to endow every arrow algebra with the structure of a BCO inducing the same tripos. A further generalization in this direction is given by Frey's \emph{uniform preorders} \cite{frey2013, frey2024}. Closely related to implicative algebras are instead Cohen, Miquey and Tate's \emph{evidenced frames} \cite{evframes2021}. We leave for further research the question of how exactly arrow algebras compare with these other structures -- now also with respect to morphisms.

\subsection*{Contents}The contents of this paper are as follows. In \autoref{sec:triposes} we briefly review the theory of triposes and their morphisms. In \autoref{sec:arralg} we summarize the theory of arrow algebras from \cite{berg2023arrow} and we describe the two main examples: those coming from frames, and those coming from partial combinatory algebras. In \autoref{sec:impmor} we introduce a first notion of morphism between arrow algebras, and in \autoref{sec:eximp1} we specialize it to the two main classes of arrow algebras described above recovering familiar notions in both cases. In \autoref{ch:arrtrip}, we show how these morphisms correspond to appropriate transformations of the associated triposes, and we introduce the notion of \emph{computational density} characterizing geometric morphisms of triposes; in \autoref{sec:examples of morphisms}, we specialize once again to the two main examples. In \autoref{sec:incl}, we focus on geometric inclusions, characterizing subtriposes of triposes arising from arrow algebras as arrow triposes themselves. This recovers and extends \cite[Sec. 6]{berg2023arrow}, in particular by proving the converse of \cite[Prop. 6.3]{berg2023arrow}, and relies on a particularly simple construction on arrow algebras which does not seem to work for implicative algebras. As subtriposes exactly correspond to subtoposes, this means in particular that we can study subtoposes of realizability toposes by studying \emph{nuclei} on the underlying arrow algebras. Finally, in \autoref{sec:mod-real} we apply the previous machinery to the study of Kreisel's modified realizability on the level of arrow algebras, rephrasing and extending results partially known in the literature at the level of PCAs.

\subsection*{2-categorical notation}Following \cite{zoethout2022}, with \emph{2-category} we mean a 2-dimensional category which is also an ordinary category, meaning that the unit and associativity laws for 1-cells hold on the nose. Instead, we speak of \emph{bicategories} for 2-dimensional categories where the axioms of an ordinary category only hold up to invertible 2-cells. 

A \emph{preorder-enriched category}, i.e. a category enriched over the category $\Preord$ of preordered sets and monotone functions, can then be seen as a locally small 2-category with at most one 2-cell between any pair of 1-cells. $\Preord$ is itself preorder-enriched with respect to the pointwise order. We refer to weak 2-functors and strict 2-functors as \emph{pseudofunctors} and \emph{2-functors}, respectively. \emph{Pseudonatural transformations} are defined in the usual way. With \emph{pseudomonad} on a preorder-enriched category, we will refer to a (fully) weak 2-monad, that is, an endo-pseudofunctor endowed with pseudonatural unit and multiplication satisfying the usual monad axioms up to isomorphism. For more details on 2-categories and 2-monads, we refer the reader to \cite{johnson20202dimensional}. 
    
In a preorder-enriched category, a morphism $f : X \to Y$ is \emph{left adjoint} to a morphism $g : Y \to X$ -- equivalently, $g$ is \emph{right adjoint} to $f$ -- if $\id_X \leq g f$ and $fg \leq \id_Y$, in which case we write $f \dashv g$. Two parallel morphisms $f,g$ are \emph{isomorphic} if $f \leq g$ and $g \leq f$, in which case we write $f \iso g$.

\section{Preliminary on tripos theory}\label{sec:triposes}
We begin by reviewing the necessary background on tripos theory, mainly drawing on the account given in \cite{van2008realizability} on the basis of  \cite{HJP80, pitts1981}; the reader is instead assumed to be familiar with topos theory, for which standard references are \cite{Lane1992SheavesIG,Borceux_1994,elephant}.

\subsection{Triposes}In this paper, we will only consider $\Sets$-based triposes.

Recall that a \emph{Heyting prealgebra} is a preorder whose poset reflection is a Heyting algebra, and a morphism of Heyting prealgebras is a monotone function which is a morphism of Heyting algebras between the poset reflections of domain and codomain. The category $\HeytPre$ of Heyting prealgebras is preorder-enriched with respect to the pointwise order.

\begin{definition}
	A \emph{($\Sets$-)tripos} is a pseudofunctor $P : \op\Sets \to \HeytPre$ satisfying the following axioms.
	\begin{enumerate}[label=\roman*.]
		\item For every function $f : X \to Y$, the map $f^* \coloneqq P(f) : P(Y)\to P(X)$ has both a left adjoint $\exists_f$ and a right adjoint $\forall_f$ in $\Preord$,\footnote{That is, $\exists_f$ and $\forall_f$ need not preserve the Heyting structure.}  and they satisfy the Beck-Chevalley condition.
		\item There exists a \emph{generic element} in $P$, i.e. an element $\sigma \in P(\Sigma)$ for some set $\Sigma$ with the property that, for every set $X$ and every element $\phi \in P(X)$, there exists a function $[\phi] : X \to \Sigma$ such that $\phi$ and $[\phi]^*(\sigma)$ are isomorphic elements of $P(X)$.
  	\end{enumerate}

A tripos $P$ such that $P(X)$ is the set $\Sigma^X$ of functions $X \to \Sigma$ for some set $\Sigma$ and $P(f)$ acts by precomposing with $f$ is said to be \emph{canonically presented}.
\end{definition}

\begin{example}\label{h valued predicates}
	Let $H$ be a complete Heyting algebra. We define the $\Sets$-tripos of $H$-valued predicates $P_H$ as follows.
 
For every set $X$, we let $P_H(X) \coloneqq H^X$, which is a Heyting algebra under pointwise order and operations; for every function $f :X\to Y$, the precomposition map $f^* : P_H(Y)\to P_H(X)$ is then a morphism of Heyting algebras. Adjoints for $f^*$ are provided by completeness as, for $\phi \in P_H(X)$ and $y\in Y$:
\[ \exists_f (\phi)(y)  \coloneqq  \biglor_{x \in \inv f (y)} \phi(x) \qquad \forall_f (\phi)(y)  \coloneqq  \bigland_{x \in \inv f(y) } \phi(x)\]
	which also satisfy the Beck-Chevalley condition. A generic element is trivially given by $\id_H \in P_H(H)$.
\end{example}

\begin{definition}
	Let $P$ and $Q$ be triposes. A \emph{transformation} $\Phi : P \to Q$ is a pseudonatural transformation $P \implies Q$, seeing $P$ and $Q$ as pseudofunctors ${\op\Sets\to \Preord}$; explicitly, this means that each component $\Phi_X : P(X) \to Q(X)$ is an order-preserving function but not necessarily a morphism of Heyting prealgebras. 
 
 Transformations $P \to Q$ can be ordered by letting $\Phi \leq \Psi$ if $\Phi_X \leq \Psi_X$ pointwise for every set $X$, therefore making triposes and transformations into a preorder-enriched category which we denote as $\Trip \Sets$. 

 A transformation $\Phi : P \to Q$ is an \emph{equivalence} if there exists another transformation $\Psi : Q \to P$ such that $\Phi \circ \Psi \iso \id_Q$ and $\Psi \circ \Phi \iso \id_P$.
\end{definition}

Every tripos $P$ gives rise to a topos $\Sets[P]$ through the \emph{tripos-to-topos construction}, which we will not describe here as it does not serve any purpose for the sake of this paper. The interested reader can find all the details in \cite{van2008realizability}.

\begin{example}
    For a complete Heyting algebra $H$, $\Sets[P_H]$ is the topos of $H$-valued sets, proved in \cite{higgs73} to be equivalent to the topos $\sh H$ of sheaves over $H$.
\end{example}

\subsection{Geometric morphisms of triposes}
The most important notion of morphism between toposes is arguably that of a \emph{geometric morphism}, which by now has a vast and standard theory. Much more niche, instead, is the theory of geometric morphisms of triposes, and how they relate with geometric morphisms of toposes: with no aim for a complete treatment, we review here what we will need in the following.

\begin{definition}\label{geometric morphism of triposes}
Let $P$ and $Q$ be triposes. A transformation $\Phi : P \to Q$ is \emph{cartesian} if each component $\Phi_X : P(X) \to Q(X)$ preserves finite meets up to isomorphism. We denote with $\Tripcart \Sets$ the wide subcategory of $\Trip \Sets$ on cartesian transformations.

A transformation is \emph{geometric} if it is cartesian and admits a right adjoint in $\Trip\Sets$, that is, if there exists another transformation $\Phi_+ : Q \to P$ such that $(\Phi^+)_X \dashv (\Phi_+)_X$ in $\Preord$ for every set $X$. We denote with $\Tripgeom \Sets$ the wide subcategory of $\Tripcart \Sets$ on geometric transformations.

In this language, a \emph{geometric morphism} of triposes $(\Phi^+, \Phi_+) : Q \to P$ is a geometric transformation $\Phi^+ : P \to Q$ with right adjoint $\Phi_+ : Q \to P$, of which they constitute respectively the \emph{inverse} and \emph{direct image}.\footnote{The direction is conventional and agrees with the definition of geometric morphisms of toposes.} 
\end{definition}

\begin{remark}
    Cartesian transformations preserve the interpretation of \emph{cartesian logic}, the fragment of finitary first-order logic defined by $\top$ and $\land$. 

    Geometric transformations preserve the interpretation of \emph{geometric logic}, the fragment of infinitary first-order logic defined by $\top$, $\bot$, finitary $\land$, infinitary $\lor$ and $\exists$.

    We will not go into details on the internal logic of triposes as it will not play any role in the paper; once again, we refer the reader to  \cite{van2008realizability}.
\end{remark}

\begin{remark}
Let $\Phi : P \to Q$ be an {equivalence} and let $\Psi : Q \to P$ be such that $\Phi \circ \Psi \iso \id_{Q}$ and $\Psi \circ \Phi \iso \id_{P}$. Then, $(\Phi,\Psi) : Q\to P$ and $(\Psi,\Phi) : P \to Q$ are both geometric morphisms. 
\end{remark}

\begin{theorem}
Every geometric morphism of triposes $Q \to P$ induces a geometric morphism $\Sets[Q] \to \Sets[P]$.
\end{theorem}

\begin{remark}
  The converse is not true in general:  a geometric morphism $\Sets[Q] \to \Sets[P]$ is induced by a geometric morphism of triposes $Q \to P$ if and only if its inverse image part \emph{preserves constant objects}. 
\end{remark}

\begin{example}
Let $X,Y$ be two complete Heyting algebras regarded as locales. Then, geometric morphisms $P_X \to P_Y$ correspond to locale homomorphisms $X \to Y$.

More precisely, given any geometric morphism $\Phi = (\Phi^+, \Phi_+) : P_X \to P_Y$, there exists an essentially unique morphism of locales $f : X \to Y$ such that, regarding $f$ as a morphism of frames $f^* : \O(Y) \to \O(X)$ and letting $f_* : \O(X) \to \O(Y)$ be its right adjoint, $\Phi^+$ is given by postcomposition with $f^*$ and $\Phi_+$ is given by postcomposition with $f_*$.
\end{example}

\subsection{Subtriposes}Geometric inclusions and surjections of toposes admit analogs on the level of triposes.

\begin{definition}
    A geometric morphism of triposes $\Phi = (\Phi^+, \Phi_+) : Q \to P$ is an \emph{inclusion}, in which case we also write $\Phi : Q \inc P$, if either of the following equivalent conditions hold:
\begin{itemize}
    \item[--]for every set $X$, $(\Phi_+)_X$ reflects the order;
    \item[--]$\Phi^+ \circ \Phi_+ \iso \id_{Q}$,
\end{itemize}

Dually, $\Phi$ is a \emph{surjection}, in which case we also write $\Phi : Q \epi P$, if either of the following equivalent conditions hold:
\begin{itemize}
    \item[--]for every set $X$, $(\Phi^+)_X$ reflects the order;
    \item[--]$\Phi_+ \circ \Phi^+ \iso \id_{P}$.
\end{itemize}
\end{definition}

\begin{proposition}
    Every geometric inclusion (resp. surjection) of triposes $Q \to P$ induces a geometric inclusion (resp. surjection) $\Sets[Q]\to \Sets[P]$.
\end{proposition}
\begin{definition}
    Let $\SubTrip {P}$ be the set of \emph{subtriposes} of $P$, that is, triposes endowed with a geometric inclusions into $P$.\footnote{For practical reasons, we identify a subtripos with the inclusion itself.} Given two geometric inclusions $\Phi : Q \inc P$ and $\Psi: R \inc P$, we write $\Phi \subseteq \Psi$ if there exists a geometric morphism $\Theta : Q \to R$ such that $\Phi \iso \Psi \circ \Theta$ -- meaning that $\Phi_+ \iso \Psi_+ \circ \Theta_+$ or equivalently $\Phi^+ \iso \Theta^+ \circ \Psi^+$ --, in which case $\Theta$ is an inclusion itself. This relation obviously makes $\SubTrip{P}$ into a preorder.

Two subtriposes $\Phi : Q \inc P$ and $\Psi: R \inc P$ are \emph{equivalent} if they are isomorphic elements of $\SubTrip P$, that is, if both $\Phi \subseteq \Psi$ and $\Psi \subseteq \Phi$ hold; equivalently, this means that there exists an equivalence $\Theta : Q \to R$ such that $\Phi \iso \Psi \circ \Theta$.
\end{definition}

As it is known, subtoposes of a topos $\E$ correspond up to equivalence to \emph{local operators} in $\E$, that is, morphisms $j : \Omega \to \Omega$ such that, in the internal logic of $\E$:
\begin{enumerate}[label=\roman*.]
    \item $j(\operatorname{t}) = \operatorname{t}$;
    \item $jj = j$;
    \item $j(a \land b) = j(a) \land j(b)$,
\end{enumerate}

In a topos of the form $\Sets[P]$ for a canonically presented tripos $P \coloneqq\Sigma^{-}$, such a morphism corresponds to an essentially unique transformation $\Phi_j : P \to P$ which is:
\begin{enumerate}[label=\roman*.]
    \item cartesian;
    \item \emph{inflationary}, that is, $\id_P \leq \Phi_j$;
    \item \emph{idempotent}, that is, $\Phi_j \Phi_j \iso \Phi_j$.
\end{enumerate}
Such $\Phi_j$ is called a \emph{closure transformation} on $P$; conversely, every closure transformation on $P$ determines a local operator on $\Sets[P]$.

These correspondences lead to the following result.

\begin{theorem}\label{thm:subtoposes and subtriposes}
Let $P$ be a canonically presented tripos and let $\ClTrans P$ be the set of closure transformations on $P$, ordered as above.
\begin{enumerate}
\item     Geometric inclusions into $P$ correspond, up to equivalence, to closure transformations on $P$; in particular, there is an equivalence of preorder categories:
\[\SubTrip P \iso \op {\ClTrans P}\]

    \item     Every geometric inclusion of toposes into $\Sets[P]$ is, up to equivalence, of the form $\Sets[Q] \inc \Sets[P]$, induced by an essentially unique geometric inclusion of triposes $Q \inc P$; in particular, there is an equivalence of preorder categories:
    \[\SubTop {\Sets[P]} \iso \SubTrip P \]
\end{enumerate}
\end{theorem}

\begin{remark}
    Note then that the poset reflection of $\SubTrip P$ is a bounded distributive lattice, since so is the set of subtoposes of any topos considered up to equivalence.
\end{remark}

In the case of a canonically presented tripos $P \coloneqq \Sigma^{-}$, we can even give an explicit description of the inclusion $Q \inc P$ inducing a geometric inclusion into $\Sets[P]$. 

Let $(\Sets[P])_j$ be the subtopos of $\Sets[P]$ corresponding to a closure transformation $\Phi_j : P \to P$ and let $J \coloneqq (\Phi_j)_\Sigma(\id_\Sigma): \Sigma \to \Sigma$. Then, $(\Sets[P])_j$ is equivalent over $\Sets$ to $\Sets[P_j]$, where $P_j$ is the canonically presented tripos defined as follows:
    \begin{itemize}
        \item[--] the underlying pseudofunctor is still $\Sigma^{-}$;
        \item[--] the order $\vdash^j$ is redefined as $\phi \vdash_I^j \psi$ if and only if $\phi \vdash_I J \psi $;
        \item[--] the implication $\arr_j$ is redefined as
        \[\begin{tikzcd}[column sep = large]
            {\Sigma \times \Sigma} \rar["{\id_{\Sigma} \times J}"] & \Sigma\times \Sigma \rar["\to"] & \Sigma
        \end{tikzcd}\]
        while $\top,\bot,\land,\lor$ remain unchanged.\footnote{Left and right adjoints for $f^*$ can then be defined as $\phi \mapsto \exists_f(\phi)$ and $\phi \mapsto \forall_f (J\phi)$.}
    \end{itemize}
This means that we can restate the previous theorem as follows.

\begin{corollary}\label{standard inclusion}
Let $P$ be a canonically presented tripos.

Every geometric inclusion of toposes into $\Sets[P]$ is induced, up to equivalence, by a geometric inclusion of triposes of the form:
\[\begin{tikzcd}
	{P_{j}} & {P}
	\arrow[""{name=0, anchor=center, inner sep=0}, "{\id_{\Sigma} \circ - }"', curve={height=12pt}, from=1-2, to=1-1]
	\arrow[""{name=1, anchor=center, inner sep=0}, "{J \circ - }"', curve={height=12pt}, from=1-1, to=1-2]
	\arrow["\dashv"{anchor=center, rotate=-90}, draw=none, from=0, to=1]
\end{tikzcd} \]
for some $J : \Sigma \to \Sigma$ corresponding as above to a closure transformation $\Phi_j$ on $P$.
\end{corollary}

\begin{example}
Let $X$ be a complete Heyting algebra regarded as a locale. Then, closure transformations on $P_X$ correspond to \emph{nuclei} on $X$, that is, monotone, inflationary and idempotent endofunctions on the underlying frame of $X$; therefore, they also correspond to \emph{sublocales} of $X$, defined dually as quotient frames of closed elements.
\end{example}

Another important notion from topos theory which can be recovered at the level of triposes is that of \emph{open} and \emph{closed} subtoposes.

\begin{definition}\label{open subtriposes}
      A subtripos $\Phi : Q \inc P$ is \emph{open} if there exists an element $\alpha \in P(1)$ such that, for every set $I$ and every $\phi \in P(I)$: 
  \[\Phi_+\Phi^+(\phi)\iso P(!)(\alpha) \arr \phi\]
  where $!$ is the unique function $I \to 1$. 

Dually, a geometric inclusion of triposes $\Psi : R \inc P$ is \emph{closed} if there exists an element $\beta \in P(1)$ such that, for every $\phi \in P(I)$:
\[\Psi_+\Psi^+(\phi)\iso  \phi \lor P(!)(\beta)\]
where $!$ is the unique function $I \to 1$.

For $\alpha = \beta$, $\Phi$ and $\Psi$ define each other's {complement} in the lattice of subtriposes of $P$ considered up to equivalence. 
\end{definition}

\begin{corollary}
    Through the correspondence in \autoref{thm:subtoposes and subtriposes}, open (resp. closed) subtriposes correspond to open (resp. closed) subtoposes.
\end{corollary}

\begin{example}
Let $X$ be a complete Heyting algebra regarded as a locale. Then, open (resp. closed) subtriposes of $P_X$ correspond to \emph{open} (resp. \emph{closed}) \emph{sublocales} of $X$.
\end{example}

\section{Arrow algebras}\label{sec:arralg}
\subsection{Arrow algebras}We now briefly review part of the theory of arrow algebras presented in \cite{berg2023arrow} so as to fix notations and make some remarks.

An \emph{arrow algebra} $\A = (A,\evleq,\arr,S)$ is the datum of a complete meet-semilattice $(A, \evleq)$, a binary operation $\arr : A \times A \to A$ called \emph{implication} which is monotone in the second component and antitone in the first component, and a specified subset $S\subseteq A$ called \emph{separator} which is upward closed, closed under modus ponens, and contains appropriate \emph{combinators} $\bf{k}$, $\bf{s}$, $\bf{a}$.  

Although an arrow algebra $\A$ is defined in terms of the \emph{evidential order} $\evleq$, there is another order that can be defined in terms of implications and separators, that is, the \emph{logical} order:
\[a \vdash b \iff a \arr b \in S\]
$(A,\vdash)$ is a Heyting prealgebra, with $\arr$ being the Heyting implication. Through the logical order, we can also recover the separator as $\Set{ a \in A | \top \vdash a}$: in hindsight, this characterizes the separator as what Pitts calls the set of \emph{designated truth values} of the induced arrow tripos.

For every set $I$, the set $A^I$ of functions $I \to A$ can be made into an arrow algebra $\A^I$ by choosing pointwise order and implication, and with the separator:
\[S^I \coloneqq \Set{ \phi : I \to A | \evmeet_{i\in I}\phi(i) \in S}\]
The logical order in $\mathcal{A}^I$, which we denote as $\vdash_I$, is thus given explicitly by:
\[\phi \vdash_I \psi \iff \evmeet_{i\in I} \phi(i)\arr \psi(i) \in S \]
In general, note then that $\vdash_I$ is stronger than the pointwise version of $\vdash$: the two orders coincide only if the separator is closed under arbitrary meets. However, it is easy to see that meets, joins and implications in $(A^I,\vdash_I)$ are computed pointwise as meets, joins and implications in $(A, \vdash)$. Given a nucleus $j$ on $\A$, we denote with $\vdash^j_I$ the logical order in $\A_j^I$: again by the properties of nuclei and separators, $\phi \vdash^j_I \psi$ explicitly means $\phi \vdash_I j \psi$.

$\A$ induces the \emph{arrow tripos}:
\[P_{\A} : \op\Sets\to \HeytPre \qquad \begin{tikzcd}
		I & {(A^I, \vdash_I)} \\
		J & {(A^J, \vdash_J)}
		\arrow["f", from=2-1, to=1-1]
		\arrow["{-\circ f}", from=1-2, to=2-2]
		\arrow[maps to, from=1-1, to=1-2]
		\arrow[maps to, from=2-1, to=2-2]
	\end{tikzcd}\]
We denote with $\AT \A$ the corresponding \emph{arrow topos} $\Sets[P_{\A}]$.

\begin{remark}
    In the following, we will make use of a form of reasoning internal to an arrow algebra $\A = (A,\evleq,\arr,S)$, which is given by a combination of \cite[Prop. 3.5]{berg2023arrow} and \cite[Lem. 3.6]{berg2023arrow} and which we will refer to as \emph{intuitionistic reasoning}. 

Suppose we want to prove that $\evmeet_{a_1,\dots,a_n\in A} \psi(a_i/p_i) \arr \chi(a_i/p_i) \in S$ for some propositional formulas $\psi,\chi$ built from propositional variables $p_1,\dots,p_n$ using implications only. Then, we will look for a propositional formula $\phi$ also built from the same propositional variables using implications only, such that
 \begin{enumerate}
        \item $\phi \arr \psi \arr \chi$ is an intuitionistic tautology, 
        \item and $\evmeet_{a_1,\dots,a_n\in A}\phi(a_i/p_i) \in S$.
\end{enumerate}
    If such a $\phi$ can be found, then by \cite[Prop. 3.5]{berg2023arrow} and (1) it will follow that:
    \[ \evmeet_{a_1,\dots,a_n \in A} \phi(a_i/p_i) \arr \psi(a_i/p_i) \arr \chi(a_i/p_i) \in S\]
    and hence, by \cite[Lem. 3.6]{berg2023arrow} and (2):
    \[ \evmeet_{a_1,\dots,a_n \in A} \psi(a_i/p_i) \arr \chi(a_i/p_i) \in S\]
\end{remark}

Let us now introduce the two main classes of arrow algebras that we will keep on analyzing throughout the paper.

\subsection{Frames}

Every frame $\O(X)$ can be canonically seen as an arrow algebra\footnote{In particular, compatible with joins.} by using its order and its Heyting implication as the arrow structure, and $\{\top\}$ as the separator. Note then that the logical order coincides with the evidential order, since $x \arr y \in S$ if and only if $\top \leq x \arr y$, which is equivalent to $x \leq y$. 

Similarly, the logical order $\vdash_I$ on $\O(X)^I$ reduces to the pointwise order,\footnote{We will see in \autoref{prop:localic arrtrip} how this property characterizes frames among arrow algebras up to equivalence of the induced triposes.} which makes so that the arrow tripos $P_{\O(X)}$ coincides with the localic tripos induced by $\O(X)$, and hence $\AT {\O(X)}$ is equivalent to the topos $ \sh {\O(X)}$ of sheaves over $\O(X)$. 

\subsection{Partial combinatory algebras}\label{subsec:pcas}

A main example of arrow algebras arises from PCAs, building blocks of realizability toposes. In this paper, we will continue with the not-entirely-standard definitions and conventions about PCAs given in \cite{berg2023arrow}, which closely follow those of \cite{zoethout2022}. This means that, with PCA, we will always mean Hofstra's notion of a \emph{filtered ordered partial combinatory algebra} \cite{hofstra2006}.   This allows us for the highest level of generality considered in the literature for what concerns appropriate notions of morphisms and their connections with morphisms of realizability triposes and toposes, as we will see in the next sections.

Let $\mathbb{P}=(P, \leq, \cdot, P^{\#})$ be a PCA and consider the poset $DP$ of downward-closed subsets of $P$. Recall from \cite[Ex. 2.1.7]{zoethout2022} that $DP$ can be given an application operation by defining, for $\alpha , \beta \in DP$:
\[\alpha \cdot \beta \coloneqq \downl{ \Set{ x  y | x \in \alpha, \ y \in \beta} } \] 
in case $\downr{x y}$ for all $x\in \alpha$ and $y\in \beta$. Then, as seen in \cite[Ex. 2.1.18]{zoethout2022}, the family of downward-closed subsets containing an element from $P^{\#}$ is a filter on this partial applicative poset:
\begin{align*}
(DP)^{\#}\coloneqq{} & \Set{ \alpha \in DP | \exists a \in \alpha \cap P^{\#} } \\
    ={} &  \Set{ \alpha \in DP | \exists a \in P^{\#} : \downl{\{a\}} \subseteq \alpha  }\\
    ={} & \Set{ \alpha \in DP | \exists \beta \in D(P^{\#}) : \exists b \in \beta \land \beta \subseteq \alpha }
\end{align*}
and two combinators $\bf{k}, \bf{s} \in P^{\#}$ for $\mathbb{P}$ yield corresponding combinators $\downl{\{\bf{k}\}}, \downl{\{\bf{s}\}} \in (DP)^{\#}$ making $(DP, \subseteq, \cdot, (DP)^{\#})$ into a PCA which we denote as $D\mathbb{P}$. In particular, for a \emph{discrete} and \emph{absolute} PCA $\mathbb{P}$, we denote $D \mathbb{P}$ also as $\Pow {\mathbb P}$.

Defining, for $\alpha,\beta \in DP$:
\[ \alpha \arr \beta \coloneqq \Set{ a \in P | \downr{a\cdot \alpha} \mbox{ and }a\cdot \alpha \subseteq \beta} \]
and letting $S_{DP} = (DP)^{\#}$, \cite[Thm. 3.9]{berg2023arrow} shows how $(DP, \subseteq, \arr, S_{DP})$ is an arrow algebra which is compatible with joins. We denote this arrow algebra as $D\mathbb{P}$ as well, and in particular as $\Pow {\mathbb P}$ in the discrete and absolute case.

This definition makes so that the arrow tripos $P_{D\mathbb{P}}$  coincides with the realizability tripos induced by $\mathbb{P}$, and hence $\AT {D \mathbb{P}}$ coincides with the realizability topos $\RT {\mathbb{P}}$.

\section{A first category of arrow algebras}\label{sec:impmor}

In this section, we introduce the first notion of a morphism between arrow algebras we will see in this paper, namely that of \emph{implicative morphisms}, and we set up some first results which will be useful in the following.

By definition, an arrow algebra is a poset endowed with an implication operation and a specified subset: therefore, it would be natural to define morphisms of arrow algebras as monotone functions preserving implications (in some sense) and the specified subset. This intuition, obviously also valid for implicative algebras, is what leads to the definition of \emph{applicative morphisms} in \cite{ferrer2019category}, which partially inspires our definition. However, for reasons which will become clear in the following, we will not define our morphisms to be monotone with respect to the evidential order,\footnote{In a way, the evidential order is not the most important feature of an arrow algebra: arrow triposes, in fact, are defined by the Heyting prealgebras determined by the logical order.} but we will see how this will not actually be an issue. The downside is that, in general, we will have to impose a third condition -- automatically satisfied in case of monotonicity -- involving both implications and separators.

\subsection{Implicative morphisms}	Let $\mathcal{A} = (A, \evleq, \arr, S_A)$ and $\mathcal{B} = (B, \evleq,\arr, S_B)$ be two arrow algebras. 

\begin{definition}\label{def:implicative morphism of aas}
An \emph{implicative morphism} $f : \mathcal{A}\to \mathcal{B}$ is a function $f :A \to B$ satisfying:
	\begin{enumerate}[label=\roman*.]
		\item $f(a) \in S_B$ for every $a \in S_A$;
		\item there exists an element $r \in S_B$ such that $r \evleq f(a\arr a')\arr f(a)\arr f(a')$ for all ${a,a'\in A}$;
		\item for every subset $X \subseteq A \times A$,
        \[\mbox{if }\evmeet_{(a,a') \in X} a \arr a' \in S_A\; \mbox{ then } \evmeet_{(a,a') \in X} f(a) \arr f(a') \in S_B ,\]
	\end{enumerate}
	in which case we say that $f$ is \emph{realized} by $r \in S_B$.

 An order `up to a realizer' can be defined on implicative morphisms as follows. Given two implicative morphisms $f,f' : \mathcal{A}\to \mathcal{B}$, we write $f \vdash f'$ if there exists an element $u \in S_B$ such that $u \evleq f(a) \arr f'(a)$ for every $a\in A$, in which case we say that $f\vdash f'$ is \emph{realized} by $u$. In other words, this means that:
\[\evmeet_{a\in A} f(a) \arr f'(a) \in S_B\]
i.e. $f \vdash_A f'$ seeing $f$ and $f'$ as elements of the arrow algebra $\B^A$, so in particular it is also equivalent to $f\phi \vdash_I f'\phi$ for every set $I$ and every function $\phi : I \to A$.
\end{definition}

\begin{remark}\label{rem:monotone morphism}
If $f$ happens to be monotone with respect to the evidential order, then (iii) is a consequence of (i) and (ii). 

Indeed, given any $X\subseteq A \times A$ such that $\evmeet_{(a,a') \in X} a\arr a' \in S_A$:
\begin{itemize}
    \item[--]by (ii) we have:
    \[\evmeet_{(a,a') \in X} f(a \arr a')\arr f(a) \arr f(a') \in S_B;\]
    \item[--]as $f(\evmeet P) \evleq \evmeet f(P)$ for every subset $P\subseteq A$ by monotonicity, and by (i) and upward-closure of $S_B$:
    \[f\left(\evmeet_{(a,a')\in X} a\arr a'\right) \evleq \evmeet_{(a,a') \in X} f(a\arr a') \in S_B,\]
\end{itemize}
from which $\evmeet_{(a,a')\in X} f(a) \arr f(a') \in S_B$  by \cite[Lem. 3.6]{berg2023arrow}.

Therefore, to prove that a monotone function is an implicative morphism, we will systematically omit to check condition (iii).
\end{remark}

\begin{remark}
    Applicative morphisms of \cite{ferrer2019category} only satisfy condition (ii) for $a,a'\in A$ such that $a\vdash a'$, while they are monotone by definition. The two notions are hence incomparable in general.
\end{remark}

\begin{proposition}
	Arrow algebras, implicative morphisms and their order form a preorder-enriched category $\ArrAlg$. 
\end{proposition}
\begin{proof}
First, let $f : \A \to \B$ and $g : \B \to \mathcal{C}$ be implicative morphisms; let us show that $gf : A \to C$ satisfies the definition of an implicative morphism $\A \to \mathcal{C}$. Condition (i) and (iii) are clearly compositional; to show condition (ii), instead, note that by (ii) for $f$ we know that:
\[\evmeet_{a,a'} f(a\arr a') \arr f(a) \arr f(a') \in S_B\]
from which, by (iii) for $g$:
\[\evmeet_{a,a'} gf(a\arr a') \arr g(f(a) \arr f(a')) \in S_C\]
Moreover, by (ii) for $g$ we know that:
\[\evmeet_{a,a'} g(f(a) \arr f(a')) \arr gf(a) \arr gf(a') \in S_C\]
from which, by intuitionistic reasoning:
\[\evmeet_{a,a'} gf(a\arr a') \arr gf(a) \arr gf(a') \in S_C\]

Then, for every arrow algebra $\A$, the identity function $\id_{A}$ is an implicative morphism $\A \to \A$, trivially realized by $\bf{i} \in S_A$ since we know that: 
     \[\bf{i} \evleq \evmeet_{a,a' \in A} (a\arr a') \arr a \arr a'\]     
This makes $\ArrAlg$ into a category.   The fact that $\vdash$ is a preorder on each homset $\ArrAlg(\A,\B)$ follows immediately as it is the subpreorder of $(B^A, \vdash_A)$ on implicative morphisms. Therefore, to conclude, we simply need to show that composition of implicative morphisms is order-preserving:
    \begin{itemize}
        \item[--]for  $f,f' : \mathcal{A}\to\mathcal{B}$ and $g : \mathcal{B}\to \mathcal{C}$ such that $f \vdash f'$; explicitly, this means that:
        \[\evmeet_{a\in A} f(a) \arr f'(a) \in S_B\]
        from which, by (iii) in \autoref{def:implicative morphism of aas}:
        \[\evmeet_{a\in A} gf(a) \arr gf'(a) \in S_C\]
        meaning that $g f \vdash g f'$;

        \item[--]for $f : \mathcal{A}\to \mathcal{B}$ and $g,g' : \mathcal{B}\to \mathcal{C}$, any realizer of $g \vdash g'$ also realizes $gf \vdash g'f$.
    \end{itemize}
\end{proof}

\begin{example}
    In a constructive metatheory, truth values are arranged in the frame $\Omega$ given by the powerset of the singleton $\{*\}$,\footnote{$\Omega$ is the initial object in the category of frames.} which we can see as an arrow algebra in the canonical way. For every arrow algebra $\A = (A, \evleq, \arr, S)$, we can then consider the characteristic function of the separator, defined constructively as:
\[\chi : A \to \Omega \qquad \chi(a) \coloneqq \Set{ * | a \in S } \]
Note that, by upward closure of the separator, $\chi$ is monotone. Indeed, if $a \evleq a'$, to show that $\chi(a) \subseteq \chi(a')$ suppose that $* \in \chi(a)$; then, $a \in S$, hence $a' \in S$ as well, i.e. $* \in \chi(a')$.

We then have that $\chi$ is an implicative morphism $\A \to \Omega$. 
\begin{enumerate}[label=\roman*.]
    \item If $a \in S$, then by definition $*\in \chi(a)$, which means that $\chi(a) = \{*\}$.
    \item Let $a,a'\in A$. Then, $\{*\} \subseteq \chi(a\arr a') \arr \chi(a) \arr \chi(a')$ is equivalent to $\chi(a\arr a') \subseteq \chi(a) \arr \chi(a')$. To show this, suppose $* \in \chi(a\arr a')$, meaning that $a\arr a' \in S$. So, $\chi(a\arr a') = \{*\}$, which means that we can show equivalently that $\chi(a) \subseteq \chi(a')$. Suppose then $*\in \chi(a)$ as well, meaning that $a \in S$; by modus ponens, it follows that $a' \in S$, i.e. $* \in \chi(a')$.
\end{enumerate}
\end{example}
    
The definition of an implicative morphism can be restated purely in terms of the logical order.

\begin{lemma}\label{lem:alt def of implicative morphism}
     Let $\A = (A,\evleq,\arr,S_A)$ and $\B = (B,\evleq, \arr, S_B)$ be arrow algebras. A function $f : A \to B$ is an implicative morphism $\A \to \B$ if and only if it satisfies:
\begin{enumerate}
    \item $\top \vdash f(\top)$;
    \item $f  (\pi_1 \arr \pi_2) \vdash_{A\times A} f \pi_1 \arr f  \pi_2$, where $\pi_1,\pi_2 : A \times A \to A$ are the two projections;
    \item $f \phi \vdash_I f \psi$ for every set $I$ and all $\phi, \psi : I \to A$ such that $\phi \vdash_I \psi$.
\end{enumerate}
\end{lemma}
\begin{proof}
Condition (2) is a rewriting of condition (ii) recalling that the Heyting implication in $\A^{A\times A}$ is computed pointwise, and condition (3) is a rewriting of condition (iii).

         Suppose now $f$ satisfies (i), (ii) and (iii). Then, $f(\top) \in S_B$ since $\top \in S_A$, which means that $\top \vdash f(\top)$. 

         Conversely, suppose that $f$ satisfies (1), (2) and (3). Note that condition (3) implies that $f$ is monotone with respect to the logical order: therefore, for $a\in S_A$ we have that $\top \vdash a$, and hence $f(\top) \vdash f(a)$. Then, $\top \vdash f(\top)$ implies by modus ponens that $f(\top) \in S_B$, from which $f(a) \in S_B$ as well.
\end{proof}

\begin{corollary}\label{cor:iso to an imp mor}
    If $f : \A \to \B$ is an implicative morphism and $f' : A \to B$ is such that $f \liso_A f'$ in $\B^A$, then $f'$ is an implicative morphism $\A \to \B$ as well.
\end{corollary}

Many implicative morphisms we will see in the following are monotone with respect to the evidential order: as it turns out, this can always be assumed up to isomorphism. Therefore, in principle, we could substitute $\ArrAlg$ for an equivalent category where all morphisms are monotone; however, we will not go in this direction, for reasons which will become clear in the next sections.

On any arrow algebra $\A$ we can consider the map:
\[ \partial : A \to A \qquad \partial a \coloneqq \top \arr a \]
By \cite[Prop. 5.8]{berg2023arrow} we have that $\partial \liso_A \id_A$; in particular, by \autoref{cor:iso to an imp mor}, $\partial$ is an implicative morphism $\A \to \A$.

\begin{lemma}\label{lem:monotonicity}
    Every implicative morphism is isomorphic to a monotone one.
\end{lemma}
\begin{proof}
Let $f : \A \to \B$ be an implicative morphism and consider the monotone function:
        \[f ' : A \to B \qquad f'(a) \coloneqq \evmeet_{a \evleq a'} \partial f(a')\]
        Let us show that $f\liso_A f'$, which in particular implies that $f'$ is an implicative morphism $\A \to \B$ by the previous corollary. 

        On one hand, $\partial \vdash_B \id_B$ gives:
        \[ \evmeet_a (\partial f(a) ) \arr f(a) \in S_B \]
        from which, since $\evmeet_{a\evleq a'} \partial f(a') \evleq \partial f(a)$ and by upward-closure of $S_B$:
        \[ \evmeet_a (\evmeet_{a \evleq a'} \partial f(a') ) \arr f(a) \in S_B\]
        i.e. $f' \vdash_A f$.
            
        On the other hand, $f \vdash_A f'$ explicitly reads as:
          \[  \evmeet_a f(a) \arr \left(  \evmeet_{a\evleq a'} \top\arr f(a') \right) \in S_B \]
            Note that, since $\bf{a} \in S_B$:
\[			\evmeet_{a} \left( \evmeet_{a\evleq a'}  f(a) \arr \top \arr f(a') \right) \arr f(a) \arr  \left(  \evmeet_{a\evleq a'} \top \arr f(a') \right) \in S_B
\]
Therefore, $f \vdash_A f'$ is ensured by \cite[Lem. 3.6]{berg2023arrow} if we show:
            \[\evmeet_{(a,a') \in I} f(a) \arr \partial f(a') \in S_B\]
            where $I \coloneqq \Set{(a,a') \in A \times A | a \evleq a'}$. By intuitionistic reasoning, this is ensured by:
            \begin{gather}
            \evmeet_{(a ,a')\in I} f(a) \arr f(a') \in S_B \\
			\evmeet_{(a ,a')\in I}  f(a') \arr \partial f(a') \in S_B 
            \end{gather}
            where (1) follows since $f$ is an implicative morphism and $\bf{i}\in S_A$ witnesses the fact that $\evmeet_{(a,a') \in I} a \arr a' \in S_A$, and (2) follows since $\id_B \vdash_B \partial$. 
\end{proof}

\begin{remark}\label{rem:monotonization}
Let $Mf \coloneqq f'$ be the monotone implicative morphism defined above. Then, $Mf \liso f$ immediately makes the association $f \mapsto Mf$ into a pseudofunctor ${M : \ArrAlg\to\ArrAlg}$.
\end{remark}

\section{Examples of implicative morphisms I}\label{sec:eximp1}
Let us now consider the two main classes of implicative morphisms, corresponding to two main examples of arrow algebras: those arising from frame homomorphisms, and those arising from morphisms of PCAs.

\subsection{Frames}
As we know, every frame can be canonically seen as an arrow algebra by choosing $\{\top\}$ as the separator. Then, any morphism of frames $f : \O(Y) \to \O(X)$, a (necessarily monotone) function which preserves finite meets and arbitrary joins, is also an implicative morphism. Indeed:
\begin{enumerate}[label=\roman*.]
\item $f(\top) = \top$ as $f$ preserves finite meets;
\item for $y,y'\in \O(Y)$ we know that $y \land (y\arr y') \leq y'$, so by monotonicity and meet-preservation $f(y) \land f(y\arr y') \leq f(y')$, meaning that $f(y\arr y') \leq f(y) \arr f(y')$ and therefore $\top \leq f(y\arr y') \arr f(y) \arr f(y')$.
\end{enumerate}

\begin{remark}\label{rem:lex map of frames}
    As emerges from the above reasoning, more generally we have that any function $f : \O(Y) \to \O(X)$ which preserves finite meets is an implicative morphism. 
\end{remark}

Recall moreover that $\Frm$ is preorder-enriched with respect to the pointwise order: therefore, given two frame homomorphisms $f,f' : \O(Y) \to \O(X)$, we have that $f \leq f'$ in $\Frm(\O(Y), \O(X))$ if and only if $f \vdash f'$ in $\ArrAlg(\O(Y), \O(X))$, since $\vdash_I$ coincides with the pointwise order for every set $I$. In other words, the inclusion determines a {2-functor} $\Frm \to \ArrAlg$ with the additional property that each map $\Frm(\O(Y),\O(X)) \hookrightarrow \ArrAlg(\O(Y),\O(X))$ (preserves and) reflects the order.

While obviously faithful, this inclusion is far from being full. Indeed, consider the initial frame $\Omega$ and let $\O(X)$ be a frame such that $\bot \neq\top$: then, the unique frame homomorphism $\Omega \to \O(X)$ is given by $p \mapsto \biglor \Set{\top | * \in p}$, whereas the constant function of value $\top$ is an implicative morphism $\Omega \to \O(X)$. As we will see in \autoref{sec:examples of morphisms}, frame homomorphisms coincide with \emph{computationally dense} implicative morphisms, while implicative morphisms between frames coincide simply with monotone functions preserving finite meets -- thus reversing the previous remark.

\subsection{Partial combinatory algebras}
The study of morphisms of PCAs in relation with induced morphisms of the associated triposes and toposes was initiated, in the discrete and absolute case, by John Longley in his PhD thesis \cite{longley95}. The notion of \emph{computational density} was introduced in 
\cite{hvO03} in the ordered case to characterize geometric morphisms of realizability toposes (see also \cite{johnstone2013b}, where computational density was also given a simpler reformulation), and then it was lifted to the ordered and relative case in \cite{hofstra2006} by considering morphisms of BCOs. We now briefly summarize the definitions we will need, as usual following the account given in \cite{zoethout2022}.

 Let $\AA = (A, \leq, \cdot, A^{\#})$ and $\BB = (B, \leq, \cdot, B^{\#})$ be two PCAs. 

\begin{definition}[\protect{\cite[Def. 2.2.1]{zoethout2022}}]\label{morphism of pcas}
A \emph{morphism of PCAs} $\mathbb{A}\to \mathbb{B}$ is a function $f : A\to B$ satisfying:
\begin{enumerate}[label = \roman*.]
	\item $f(a) \in B^{\#}$ for every $a \in A^{\#}$;
	\item there exists an element $t \in B^{\#}$ such that if $\downr{a a'}$ then $\downr{t f(a)f(a')}$ and $tf(a)f(a') \leq f(aa')$;
	\item there exists an element $u \in B^{\#}$ such that if $a\leq a'$ then $\downr{uf(a)}$ and $uf(a) \leq f(a')$,
\end{enumerate}
in which case we say that $f$ is \emph{realized} by $t, u\in B^{\#}$, or that it preserves application up to $t$ and order up to $u$.

An order `up to a realizer' can be defined on morphisms of PCAs as follows. Given two morphisms $f,f' : \mathbb{A}\to \mathbb{B}$, we write $f \leq f'$ if there exists some $s \in B^{\#}$ such that $\downr{sf(a)}$ and $s f(a) \leq f'(a)$ for every $a\in A$, in which case we say that $f\leq f'$ is \emph{realized} by $s$.

PCAs, morphisms of PCAs and their order form a preorder-enriched category $\oPCA$.

\end{definition}

\begin{definition}[\protect{\cite[Lem. 2.2.12]{zoethout2022}}]
	A morphism of PCAs $f : \mathbb{A}\to \mathbb{B}$ is \emph{computationally dense} if there exists an element $m\in B^{\#}$ with the property that for every $s \in B^{\#}$ there is some $r \in A^{\#}$ such that  $mf(ra)\preceq sf(a)$ for every $a \in A$.
\end{definition}

The construction $\mathbb{P} \mapsto D\mathbb{P}$ lifts to a pseudomonad on $\oPCA$ (\cite[Prop. 2.3.1]{zoethout2022}), through which a new notion of morphism of PCAs can be defined.

\begin{definition}[\protect{\cite[Def. 2.3.2]{zoethout2022}}]
Let $\oPCA_D$ be the preorder-enriched bicategory defined as the Kleisli bicategory of the pseudomonad $D : \oPCA \to \oPCA$. Explicitly, $\oPCA_D$ is the category having PCAs as objects, and morphisms of PCAs $\AA \to D \BB$ as morphisms $\AA \to \BB$, which we call \emph{partial applicative morphisms}.

A morphism in $\oPCA_D$ is \emph{computationally dense} if it is so as a morphism of PCAs. Explicitly, a partial applicative morphism $f : \AA \to \BB$ is then computationally dense if there exists an element $m\in B^{\#}$ with the property that for every $s\in B^{\#}$ there is some $r \in A^{\#}$ satisfying $\downr{mf(ra)}$ and $mf(ra)\subseteq sf(a)$ for every $a \in A$ such that $\downr{sf(a)}$. 
\end{definition}

Let us now move to the level of arrow algebras. 

First, let us show how every morphism of PCAs $D \AA \to D\BB$ is an implicative morphism with respect to the canonical arrow structures on $D\AA$ and $D\BB$.

\begin{lemma}\label{lem:applicative implicative}
    Let $f : D \AA \to D \BB$ be a morphism of PCAs. Then, $f$ is also an implicative morphism $D \AA \to D \BB$.
\end{lemma}
\begin{proof}
Let us verify the three conditions in \autoref{def:implicative morphism of aas}.
        \begin{enumerate}[label=\roman*.]
\item         Condition (i) is ensured by (i) in \autoref{morphism of pcas} since $S_{DA} = (DA)^{\#}$ and $S_{DB} = (DB)^{\#}$.

\item To show condition (ii), we need to find some $\rho \in (DB)^{\#}$ such that $\rho \subseteq {f}(\alpha\arr\beta)\arr {f}(\alpha) \arr {f}(\beta)$ for every $\alpha,\beta \in DA$. By definition, recall that:
        \begin{itemize}
            \item[--]there exists $\tau \in (DB)^{\#}$ such that if $\downr{\alpha \alpha'}$, then $\downr{\tau f(\alpha) f(\alpha')}$ and $\tau f(\alpha) f(\alpha') \subseteq f(\alpha \alpha')$;
            \item[--]there exists $\upsilon \in (DB)^{\#}$ such that if $\alpha \subseteq \alpha'$ then $\downr{\upsilon f(\alpha)}$ and $\upsilon f(\alpha) \subseteq f(\alpha')$.
        \end{itemize}
By combinatory completeness, consider then:
\[\rho \coloneqq (\lambda^* v, w . \upsilon( \tau v w ) ) \in (DB)^{\#}\]

Since $\downr{ (\alpha \arr \beta)\cdot \alpha}$, we know that:
\[\downr{\tau \ f(\alpha\arr\beta) \ f(\alpha)} \quad \mbox{and} \quad \tau \ f(\alpha\arr \beta) \ f(\alpha) \subseteq f( (\alpha\arr \beta) \cdot \alpha)\]
Since moreover $(\alpha\arr\beta)\cdot \alpha \subseteq \beta$, we also know that:
\[ \downr{\upsilon \ f( (\alpha\arr \beta) \cdot \alpha)}\quad \mbox{and} \quad \upsilon \ f( (\alpha\arr \beta) \cdot \alpha) \subseteq f(\beta) \]
So, by downward-closure of the domain of the application:
\[\downr{ \upsilon (\tau \ f(\alpha\arr\beta) \ f(\alpha))}\quad \mbox{and} \quad \upsilon (\tau \ f(\alpha\arr\beta) \ f(\alpha) ) \subseteq f(\beta)\]

Therefore: 
\[\downr{\rho \ f(\alpha\arr \beta) \ f(\alpha)} \quad \mbox{and} \quad \rho \ f(\alpha\arr \beta) \ f(\alpha)\subseteq f(\beta) \]
or, in other words:
\[\rho \subseteq f(\alpha \arr \beta) \arr f(\alpha) \arr f(\beta)\]

\item Let $X \subseteq DA \times DA$ be such that there exists some $\sigma \in (DA)^{\#}$ satisfying $\sigma \subseteq \alpha \arr \beta$ for every $(\alpha,\beta) \in X$. To show condition (iii), we need to find some $\rho \in (DB)^{\#}$ such that $\rho \subseteq f(\alpha) \arr f(\beta)$ for every $(\alpha,\beta) \in X$.

By combinatory completeness, since $f(\sigma) \in (DB)^{\#}$ by condition (i), consider then:
\[\rho \coloneqq (\lambda^* w . \upsilon (\tau f(\sigma) w )) \in (DB)^{\#}  \]
Since $\downr{\sigma \cdot \alpha}$ and $\sigma \cdot \alpha \subseteq \beta$, exactly as above we have that:
\[\downr{ \upsilon (\tau \ f(\sigma) \ f(\alpha))}\quad \mbox{and} \quad \upsilon (\tau \ f(\sigma) \ f(\alpha) ) \subseteq f(\beta)\]

Therefore:
\[\downr{\rho \ f(\alpha)} \quad \mbox{and}\quad \rho \ f(\alpha) \subseteq f(\beta) \]
or, in other words:
\[ \rho \subseteq f(\alpha) \arr f(\beta)\]
        \end{enumerate}
\end{proof}

\begin{remark}
    The two orders also coincide, by definition of the implication in downsets PCAs: given two morphisms of PCAs $f,f' : D \AA \to D \BB$, then $f \leq f'$ in $\Hom \oPCA {D\AA} {D\BB}$ if and only if $f \vdash f'$ in $\Hom \ArrAlg {D\AA} {D\BB}$.
\end{remark}

Let now $f : \AA \to \BB$ be a partial applicative morphism, i.e. a morphism of PCAs $\AA \to D \BB$. Then, $f$ corresponds to an essentially unique $D$-algebra morphism $\tilde{f} : D\AA \to D \BB$ which, up to isomorphism, we can describe as:
\[\tilde{f}(\alpha) \coloneqq \bigcup_{a\in \alpha} f(a)\]
The association $f \mapsto \tilde{f}$ is 2-functorial\footnote{As noted in \protect{\cite[Rem. 2.3.4]
{zoethout2022}}, $\oPCA_D$ is only a bicategory and not a preorder-enriched category, but since compositions are defined on the nose we can still speak of 2-functors rather than pseudofunctors.} on $\oPCA_D$, and it realizes an equivalence of preorder categories between partial applicative morphisms $\AA \to \BB$ and $D$-algebra morphisms $D\AA \to D \BB$. Together with the previous lemma and remark, this immediately implies the following.

\begin{proposition}\label{prop:tildeD}
    The assignment $f \mapsto \tilde{f}$ determines a $2$-functor $\tilde{D} :
        \oPCA_D \to \ArrAlg$.
    
    Moreover, for all PCAs $\AA$ and $\BB$, the map $\Hom{\oPCA_D}{\AA}{\BB} \to \Hom\ArrAlg{D\AA}{D\BB}$ (preserves and) reflects the order.
\end{proposition}

\begin{remark}
The maps $\Hom{\oPCA_D}{\AA}{\BB} \to  \Hom \ArrAlg {D\AA}{ D \BB}$ defined by $\tilde{D}$ are obviously not essentially surjective, meaning that $\tilde{D}$ is not 2-fully faithful. Indeed, any morphism of PCAs $D \AA \to D\BB$ is an implicative morphism $D \AA \to D\BB$, but obviously only those which are union-preserving are $D$-algebra morphisms and therefore arise as $\tilde{D}f$ for some partial applicative morphism $f : \AA \to \BB$. We will see more about the interplay of these notions in \autoref{sec:examples of morphisms}.  
\end{remark}

\begin{question}
    How does the $\operatorname{PER}(-)$ construction of \cite[Thm. 3.10]{berg2023arrow} behave with respect to (partial applicative) morphisms of PCAs and implicative morphisms?
\end{question}

\section{Transformations of arrow triposes}\label{ch:arrtrip}
We have finally arrived at the heart of this paper. In this section, we further the study of implicative morphisms and their relations with transformations of arrow triposes, lifting the association $\A \mapsto P_{\A}$ to a 2-functor defined on a suitable category of arrow algebras. The main goals, in this perspective, are the following.
\begin{enumerate}[label=\roman*.]
    \item First, we will characterize implicative morphisms $\A \to \B$ as those functions $A \to B$ which induce by postcomposition a cartesian transformation of arrow triposes $P_{\A} \to P_{\B}$. 
    \item Then, we will determine a suitable notion of \emph{computational density} which characterizes those implicative morphisms $\A\to \B$ such that the induced transformation $P_{\A} \to P_{\B}$ has a right adjoint, hence corresponding to geometric morphism of triposes $P_{\B}\to P_{\A}$.
\end{enumerate}

\subsection{Cartesian transformations of arrow triposes}
Let us start with a lemma we will make use of in the following. Recall that, given any two arrow algebras $\A = (A, \evleq, \arr, S_A)$ and $\B = (B, \evleq, \arr, S_B)$, for every set $I$ we can consider the arrow algebras 
$\A^I = (A^I, \evleq, \arr, S_A^I)$ and $\B^I = (B^I, \evleq, \arr, S_B^I)$.

\begin{lemma}
Let $f : \A \to \B$ be an implicative morphism. 

For every set $I$, $f^I \coloneqq f \circ -$ is an implicative morphism $\A^I\to \B^I$.
\end{lemma}
\begin{proof}
Let us verify the three conditions in \autoref{def:implicative morphism of aas}.
        \begin{enumerate}[label=\roman*.]
        \item To show condition (i), recall that we can equivalently prove that $f^I(\top_I) \in S_B^I$, where $\top_I:I \to A$ is the constant function of value $\top \in A$. Note then that by condition (i) for $f$ we have: 
        \[ \evmeet_i f(\top_I(i)) = \evmeet_i f(\top) = f(\top) \in S_B\]
        meaning that $f^I(\top) \in S^I_B$.
        
        \item Let $r \in S_B$ be a realizer for $f$ and let $\rho : I \to B$ be the constant function at $r$; as $\evmeet_i \rho(i) = r \in S_B$, we know that $\rho \in S^I_B$. 

        Then, for all $\phi,\phi' \in A^I$:
        \[r \evleq f(\phi(i) \arr \phi'(i) ) \arr f\phi(i) \arr f\phi'(i) \quad \forall i\in I\]
        i.e., since order and implications are defined pointwise in $\B^I$:
        \[\rho \evleq f^I(\phi \arr \phi') \arr f^I\phi \arr f^I\phi' \]
        meaning that $\rho$ realizes $f^I$.
        
        \item Let $X \subseteq A^I \times A^I$ be such that $\evmeet_{(\phi,\psi) \in X} \phi \arr \psi \in S_A^I$. For the sake of notation, assume $X = \Set{ (\phi_j, \psi_j) | j \in J}$; then, since order (hence meets) and implications are defined pointwise in $\A^I$, we have:
        \[\evmeet_i \evmeet_j \phi_j(i) \arr \psi_j(i) \in S_A\]
        from which, by (iii) for $f$:
        \[\evmeet_i \evmeet_j f\phi_j(i) \arr f\psi_j(i) \in S_A\]
        meaning that $\evmeet_j f^I\phi_j \arr f^I\psi_j \in S_A^I$.
        \end{enumerate}
\end{proof}

Fix now an implicative morphism $f : \A \to \B$ and define the transformation:
\[\Phi_f^+ : P_{\A} \to P_{\B} \qquad (\Phi_f^+)_I (\phi) \coloneqq f^I\phi = f \circ \phi\]
Indeed, monotonicity of each component $(\Phi_f^+)_I : P_{\A}(I) \to P_{\B}(I)$ precisely corresponds to condition (iii) in \autoref{def:implicative morphism of aas}, while naturality is obvious.

Let us now show that, for every set $I$, $(\Phi_f^+)_I: P_{\A}(I) \to P_{\B}(I)$ preserves finite meets up to isomorphism. As $f^I(\top_I) \in S_B^I$ we know that $f^I(\top_I) \liso_I \top_I$, so we only have to show that for all $\phi,\psi \in A^I$:
\[ f^I(\phi \land \psi) \liso_I f^I\phi \land f^I\psi \]
where $f^I\phi \land f^I\psi$ is the meet of $f^I\phi$ and $f^I\psi$ in $P_{\B}(I)$ which, as we already recalled, can be assumed to be defined pointwise. 

Of course $f^I(\phi \land \psi) \vdash_I f^I\phi \land f^I \psi$ follows simply by monotonicity of $f^I$ with respect to the logical order; on the other hand, $f^I\phi \land f^I\psi \vdash_I f^I(\phi\land \psi)$ is ensured by the following lemma applied to $f^I : \A^I \to \B^I$.

\begin{lemma}\label{lem:implicative morphisms preserve logical meets}
    Let $f : \A \to \B$ be an implicative morphism. Then:
    \[\evmeet_{a,b \in A} (f(a) \land f(b)) \arr f(a\land b) \in S_B\]
\end{lemma}
\begin{proof}
First, note that:
        \[\evmeet_{a,b\in A} a \arr b \arr (a\land b) \in S_A\]
        from which, by (iii) in \autoref{def:implicative morphism of aas}:
        \[\evmeet_{a,b\in A} f(a) \arr f(b \arr (a\land b)) \in S_B\]
        Moreover, by (ii):
        \[\evmeet_{a,b\in A} f(b \arr (a\land b))\arr f(b) \arr f(a\land b) \in S_B\]
        so by intuitionistic reasoning we conclude:
        \[\evmeet_{a,b\in A} f(a) \arr f(b) \arr f(a\land b) \in S_B\]
        which means:
        \[\evmeet_{a,b\in A} (f(a) \land f(b)) \arr f(a\land b) \in S_B\]
\end{proof}

Summing up, we have shown the following.

\begin{proposition}\label{prop:morph induces cartesian transformation of triposes}
    Let $f : \A \to \B$ be an implicative morphism. Then:
    \[\Phi_f^+ : P_{\A} \to P_{\B} \qquad (\Phi_f^+)_I (\phi) \coloneqq  f \circ \phi\]
    is a cartesian transformation of triposes. 
\end{proposition}

As promised above, we can also prove the converse: up to isomorphism, every cartesian transformation of arrow triposes is induced in the sense of the previous proposition by an implicative morphism which is unique up to isomorphism.

\begin{proposition}\label{prop:lex trans is induced by morphism}
The association $f \mapsto \Phi_f^+$ determines a 2-fully faithful 2-functor $\ArrAlg \inc \Tripcart{\Sets}$.

  Explicitly, this means that for all arrow algebras $\A$ and $\B$ there is an equivalence of preorder categories:
  \[	{\Hom \ArrAlg \A \B} \equiv {\Hom {\Tripcart\Sets }{P_{\A}} {P_{\B}} }\]
\end{proposition}
\begin{proof}
By the previous discussion, we have a functor $\ArrAlg \to \Tripcart\Sets$; 2-functoriality amounts to showing that, given any two implicative morphisms $f,f' : \A \to \B$ such that $f \vdash f'$, then $\Phi^+_f \leq \Phi^+_{f'}$. By definition, $f \vdash f'$ means that for every set $I$ and every $\phi : I \to A$ then $f\phi \vdash_I f'\phi$ holds, i.e. $(\Phi^+_f)_I(\phi) \vdash_I (\Phi^+_{f'})_I(\phi)$ holds in $P_{\B}(I)$, which means that $\Phi^+_f \leq \Phi^+_{f'}$. Note then that the converse also holds: if $\Phi_f^+ \leq \Phi_{f'}^+$, then in particular we have that $(\Phi^+_f)_A(\id_A) \vdash_A (\Phi^+_{f'})_A(\id_A)$, which means that $f\vdash f'$.

Let now $\Phi^+: P_{\A} \to P_{\B}$ be any cartesian transformation of arrow triposes. Recall that, up to isomorphism, $\Phi^+$ is given by postcomposition with the function:
    \[f \coloneqq (\Phi^+)_A(\id_A) : A \to B\]
        
        Let us now verify that $f$ satisfies the three conditions\footnote{If we assumed implicative morphisms to be monotone, we would not be able to prove that $f$ is one.} in \autoref{def:implicative morphism of aas}.
        \begin{enumerate}[label=\roman*.]
            \item To show condition (i), recall that we can equivalently prove that $f(\top) \in S_B$. By preservation of finite meets, we know that $(\Phi^+)_I(\top_I) \liso_I \top_I$, which for $I\coloneqq 1$ means that $f(\top) \dashv\vdash \top$, i.e. $f(\top)\in S_B$.
            
            \item Let $I \coloneqq A \times A$; recall that condition (ii) can be rewritten as:
    \[f (\pi_1\arr\pi_2) \vdash_{I} f \pi_1 \arr f \pi_2\]
    where $\pi_1,\pi_2 : I \to A$ are the two projections. In terms of $\Phi^+$, this means that we have to show:
    \[(\Phi^+)_I (\pi_1\arr\pi_2) \vdash_I (\Phi^+)_I (\pi_1) \arr (\Phi^+)_I (\pi_2)\]
    Through the Heyting adjunction in $P_{\B}(I)$, the previous is equivalent to:
    \[(\Phi^+)_I (\pi_1\arr\pi_2) \land (\Phi^+)_I (\pi_1) \vdash_I (\Phi^+)_I (\pi_2)\]
    i.e., by preservation of finite meets:
    \[(\Phi^+)_I (\pi_1\arr\pi_2 \land \pi_1) \vdash_{I} (\Phi^+)_I (\pi_2)\]
    which is ensured by monotonicity since $\pi_1 \arr \pi_2 \land \pi_1 \vdash_I \pi_2$.
    
\item Condition (iii) precisely corresponds to the monotonicity of each component $(\Phi^+)_I$.
\end{enumerate}

Therefore, the association $\Phi^+ \mapsto (\Phi^+)_A(\id_A)$ realizes the desired inverse equivalence since obviously $(\Phi^+_f)_A(\id_A) = f$ for every implicative morphism $f : \A \to \B$.
\end{proof}

\begin{remark}\label{rem:2-equivalence morphisms and lex trans}
In \cite{miquel2020-2} it is shown that every canonically-presented ($\Sets$-)tripos is equivalent to an implicative one, and hence to an arrow one. 

Therefore, assuming the Axiom of Choice, the 2-functor $\ArrAlg \to \Tripcart\Sets$ is actually a 2-equivalence of 2-categories: in fact, under the Axiom of Choice, every tripos is equivalent to a canonically presented one (cfr. \cite[Prop. 1.9]{HJP80}), meaning that the above 2-functor is essentially surjective on objects; moreover, under the Axiom of Choice, a 2-functor is a 2-equivalence if and only if it is 2-fully faithful and essentially surjective on objects, see e.g. \cite[Thm. 7.4.1]{johnson20202dimensional}.
\end{remark}

\subsection{Geometric morphisms of arrow triposes}Let us now move to geometric morphisms: as we will see in a moment, the existence of a right adjoint at the level of transformations of triposes can exactly be characterized by the existence of a right adjoint in $\ArrAlg$.

\begin{definition}\label{def:cd morphism of aas}
     An implicative morphism $f : \A \to \B$ is \emph{computationally dense\footnote{The name is obviously taken from the theory of PCAs, and it is also used in \cite{ferrer2019category} for applicative morphisms.}} if it admits a right adjoint in $\ArrAlg$, that is, if there exists an implicative morphism $h : \B \to \A$ such that $f h \vdash \id_B$ and $\id_A \vdash h f$.
\end{definition}

For every arrow algebra $\A$, the identity $\id_A : \A \to \A$ is computationally dense, as it is trivially right adjoint to itself. As the existence of a right adjoint in the preorder-enriched sense is clearly compositional, we can define the wide subcategory $\ArrAlgcd$ of $\ArrAlg$ on computationally dense morphisms.

Fix now a computationally dense implicative morphism $f : \A \to \B$ with right adjoint $h : \B \to \A$ and consider the cartesian transformation induced by $h$ as in \autoref{prop:morph induces cartesian transformation of triposes}:
\[\Phi_+ : P_{\B} \to P_{\A} \qquad (\Phi_+)_I (\phi) \coloneqq h\circ \phi\]

\begin{lemma}
    For every set $I$, $(\Phi_+)_I : P_{\B}(I) \to P_{\A}(I)$ is right adjoint to the map $(\Phi_f^+)_I : \psi \mapsto f\circ \psi$.
\end{lemma}
\begin{proof}
By the universal property of the counit of the adjunction, it suffices to show that for every $\phi : I \to B$:
        \begin{enumerate}
            \item $f h \phi \vdash_I \phi$ in $P_{\B}(I)$;
            \item for every $\psi : I \to A$ such that $f \psi \vdash_I \phi$ in $P_{\B}(I)$, then $\psi \vdash_I h \phi$ in $P_{\A}(I)$.
        \end{enumerate}

        (1) clearly follows since $h$ is right adjoint to $f$. To show (2), instead, suppose $f \psi \vdash_I \phi$; then, $h f \psi \vdash_I h \phi$ as $h$ is an implicative morphism, and hence $\psi \vdash_I h \psi$ since $\id_A \vdash_A h f$.
\end{proof}

The previous results then immediately yield the following.

\begin{theorem}\label{thm:induced geometric morphism}
    Let $f : \A \to \B$ be a computationally dense implicative morphism with right adjoint $h : \B \to \A$. Then, $\Phi^+_f : P_{\A} \to P_{\B}$ is a geometric transformation of triposes; in other words, the pair:
    \[\begin{tikzcd}
	{P_{\B}} & {P_{\A}}
	\arrow[""{name=0, anchor=center, inner sep=0}, "{\Phi^+}"', curve={height=12pt}, from=1-2, to=1-1]
	\arrow[""{name=1, anchor=center, inner sep=0}, "{\Phi_+}"', curve={height=12pt}, from=1-1, to=1-2]
	\arrow["\dashv"{anchor=center, rotate=-90}, draw=none, from=0, to=1]
\end{tikzcd} \qquad \begin{cases}
    (\Phi^+)_I (\psi) \coloneqq f \circ \psi \\
    (\Phi_+)_I (\phi) \coloneqq h \circ \psi
\end{cases}\]
is a geometric morphism $P_{\B}\to P_{\A}$.
\end{theorem}

As we did for implicative morphisms and cartesian transformations, in this case too we can prove the converse: up to isomorphism, every geometric morphism of arrow triposes is induced by an essentially unique computationally dense implicative morphism.

\begin{proposition}\label{prop:geo is induced by cd morphism}
    The 2-functor of \autoref{prop:lex trans is induced by morphism} restricts to a 2-fully faithful 2-functor $\ArrAlgcd \inc \Tripgeom\Sets$.

    Explicitly, this means that for all arrow algebras $\A$ and $\B$ there is an equivalence of preorder categories:
    \[ {\Hom\ArrAlgcd\A\B} \equiv {\Hom{\Tripgeom\Sets}{P_{\A}}{P_{\B}}}\]
\end{proposition}
\begin{proof}
By the previous discussion, we have a 2-functor ${\ArrAlgcd \to \Tripgeom\Sets}$ such that, given any two computationally dense implicative morphisms $f,f' : \A \to \B$, $f \vdash f'$ if and only if $\Phi_f^+ \leq \Phi_{f'}^+$.
    
        Let now $\Phi^+ : P_{\A} \to P_{\B}$ be a geometric transformation of triposes, that is, a cartesian transformation having a right adjoint $\Phi_+ : P_{\B} \to P_{\A}$. Recall that, up to isomorphism, $\Phi^+$ is given by postcomposition with $f \coloneqq (\Phi^+)_A(\id_A) : A \to B$,
        which is an implicative morphism $\A \to \B$ by \autoref{prop:lex trans is induced by morphism}. In the same way, as it is also necessarily cartesian, $\Phi_+$ is given up to isomorphism by postcomposition with the implicative morphism $h \coloneqq (\Phi_+)_B(\id_B) : \B \to \A$. Moreover, the adjunction between $\Phi^+$ and $\Phi_+$ directly yields $fh \vdash \id_B$ and $\id_A \vdash hf$, meaning that $h$ is right adjoint to $f$ making it computationally dense.
\end{proof}

\begin{remark}
    As in \autoref{rem:2-equivalence morphisms and lex trans}, assuming the Axiom of Choice, the above 2-functor $\ArrAlgcd \to \Tripgeom\Sets$ is a 2-equivalence of 2-categories.
\end{remark}

\subsection{Equivalences of arrow triposes}Finally, let us characterize equivalences of arrow triposes on the level of arrow algebras.

With usual 2-categorical notation, we say that an implicative morphism $f : \A \to \B$ is an \emph{equivalence} if there exists another implicative morphism $g : \B \to \A$ such that $f g \liso \id_B$ in $\ArrAlg(\B,\B)$ and $g f \liso \id_A$ in $\ArrAlg(\A,\A)$, in which case $g$ is a \emph{quasi-inverse} of $f$. Two arrow algebras are then \emph{equivalent} if there exists an equivalence between them; clearly, equivalent arrow algebras induce equivalent triposes.

\begin{lemma}\label{lem:equivalences}
    Let $f : \A \to \B$ be an equivalence of arrow algebras. 
    
    Then, $f$ is computationally dense, and the induced geometric morphism of triposes $\Phi : P_{\B}\to P_{\A}$ is an equivalence.
\end{lemma}
\begin{proof}
Let $g : \B \to \A$ be a quasi-inverse of $f$. As $g$ is in particular right adjoint to $f$ in $\ArrAlg$, $f$ is computationally dense, and the induced geometric morphism $\Phi = (\Phi^+, \Phi_+) :P_{\B}\to P_{\A}$ is given by:
        \[(\Phi^+)_I (\psi) = f \circ \psi \qquad (\Phi_+)_I(\phi) = g \circ \phi\]
        In particular, $\Phi^+\Phi_+$ and $\Phi_+\Phi^+$ are isomorphic to identities since $f g \liso \id_B$ and $g f \liso \id_A$, meaning that $\Phi$ is an equivalence.
\end{proof}

By the previous results, we can also easily address the converse.

\begin{proposition}\label{prop:equivalent triposes are induced by equivalences}
    Let $\Phi : P_{\A} \to P_{\B}$ be an equivalence of arrow triposes. Then, $\Phi$ is induced up to isomorphism by an (essentially unique) equivalence of arrow algebras $f : \A \to \B$.
\end{proposition}
\begin{proof}
Let $\Psi : P_{\B} \to P_{\A}$ be a quasi-inverse of $\Phi$. Then, $\Phi$ is both left and right adjoint to $\Psi$, which means in particular that the pair $(\Phi, \Psi)$ defines a geometric morphism $P_{\B} \to P_{\A}$. Therefore, by \autoref{prop:geo is induced by cd morphism}, $\Phi$ is induced up to isomorphism by an (essentially unique) computationally dense implicative morphism $f : \A \to \B$; a right adjoint $g : \B\to\A$ inducing $\Psi$ up to isomorphism then satisfies $f g \liso \id_B$ and $g f \liso \id_A$, making $f$ an equivalence.
\end{proof}

\begin{remark}\label{rem:strong arrow algebras}
    Recall from \cite{berg2023arrow} that  an arrow algebra $\A = (A, \evleq, \arr, S)$ is \emph{strong} if 
    \[ \evmeet_{a \in A, B \subseteq A} \left( \evmeet_{b\in B} a \arr b \right) \arr a \arr \evmeet_{b\in B} b  \in S \]

    \cite[Sec. 5]{berg2023arrow} then shows how every arrow algebra $\A = (A, \evleq, \arr, S)$ is equivalent to a strong one. In fact, recall by \cite[Def. 5.4]{berg2023arrow} that an element $a\in A$ is \emph{functional} if 
\[a = \evmeet_{b,c\in A} \Set{ b\arr c | a\evleq b\arr c}\]
and, letting $A_{\text{fun}}$ be the set of functional elements of $\A$ and $S_{\text{fun}} \coloneqq S \cap A_{\text{fun}}$,  \cite[Cor. 5.6]{berg2023arrow} shows how $\A_{\text{fun}} \coloneqq ( A_{\text{fun}}, \evleq, \arr, S_{\text{fun}} )$ is a strong arrow algebra. Then, since implications are functional, $\partial \liso_A \id_A$ implies that $\partial : A \to A$ defines an equivalence $\A \to \A_{\text{fun}}$, with right adjoint given by the inclusion $A_{\text{fun}} \inc A$.
\end{remark}

\begin{remark}\label{rem:trivial aa}
We say that an arrow algebra $\A = (A,\evleq,\arr,S)$ is \emph{trivial} if $S = A$, or equivalently if $\bot \in S$. It is then immediate to show that $\A$ is trivial if and only if the unique map $\A \to \{*\}$ -- which is obviously an implicative morphism -- is an equivalence. Hence, $\A$ is trivial if and only if $\AT \A$ is (equivalent to) the trivial topos.
\end{remark}

As an example of application of the above correspondence between equivalences, we can easily characterize localic triposes among arrow triposes:\footnote{We thank the referee for suggesting such a characterization.} in particular, thanks to \cite{HJP80} we also have an explicit description -- up to equivalence in $\ArrAlg$ -- of the frame inducing the tripos.

\begin{proposition}\label{prop:localic arrtrip}
    Let $\A = (A,\evleq,\arr, S)$ be an arrow algebra. Then, the following are equivalent:
    \begin{enumerate}
    \item for every set $I$, the logical order $\vdash_I$ on $\A^I$ coincides with the pointwise version of $\vdash$;
    \item the Heyting algebra $(A/{\liso}, \vdash, \arr)$ is complete, hence a frame which we see as an arrow algebra $\bar {\A}$ in the canonical way, and the projection map $\A \to \bar\A$ is an equivalence in $\ArrAlg$;
    \item $\A$ is equivalent to a frame, seen as an arrow algebra in the canonical way;
    \item $\A$ is equivalent to an arrow algebra having $\{\top\}$ as separator;
    \item $P_{\A}$ is localic.
    \end{enumerate}
\end{proposition}
\begin{proof}
$(1) \Leftrightarrow (2)$ coincides with (part of) \cite[Thm. 4.1]{HJP80} making use of \autoref{prop:equivalent triposes are induced by equivalences}.

$(2)\implies (3)$ is obvious.

$(3) \implies (4)$ is obvious.

To show $(4)\implies (1)$, suppose there exists an equivalence $f : \A \to \B$ for some arrow algebra $\B = (B, \evleq, \arr, \{\top\})$. Then, for every set $I$ and all $\phi,\psi : I \to A$:
\begin{align*}
    \phi \vdash_I \psi & \iff f \phi \vdash_I f \psi \\
        & \iff \top = \evmeet_{i\in I} f\phi(i) \arr f \psi (i)  \\
    & \iff \forall i \in I, \ \top =  f\phi(i) \arr f \psi (i) \\
    & \iff \forall i \in I, \   f\phi(i) \vdash f \psi (i) \\
    & \iff \forall i \in I, \ \phi(i) \vdash  \psi (i) 
\end{align*}

Finally, $(3) \Leftrightarrow (5)$ follows from \autoref{lem:equivalences} and \autoref{prop:equivalent triposes are induced by equivalences}.
\end{proof}

\section{Examples of implicative morphisms II}\label{sec:examples of morphisms}
We can now finally conclude our analysis of the two main classes of arrow algebras, namely those arising from frames and from PCAs, now studying their morphisms in relation to the transformations between the associated triposes.

\subsection{Frames}
First, recall that frame homomorphisms are implicative morphisms, seeing frames as arrow algebras in the canonical way. More generally, as noted in \autoref{rem:lex map of frames}, every function between frames which preserves finite meets is an implicative morphisms: we can now easily prove the converse as well. In fact, if $f : \O(Y) \to \O(X)$ is an implicative morphism, then:
\begin{itemize}
    \item[--] as the separator on a frame is canonically defined as $\{\top\}$, $f(\top) = \top$;
    \item[--] by \autoref{lem:implicative morphisms preserve logical meets}, $f$ preserves binary logical meets, but the logical order on frames coincides with the evidential order and hence $f$ preserves binary (``evidential'') meets.
\end{itemize}

Moving on, let us see how every frame homomorphism is computationally dense as an implicative morphism.

\begin{proposition}
    Let $f^* : \O(Y)\to \O(X)$ be a frame homomorphism. Then, $f^*$ is a computationally dense implicative morphism.
\end{proposition}
\begin{proof}
Let $f_* : \O(X) \to \O(Y)$ be the right adjoint of $f^*$, i.e. the monotone function:
        \[f_*(x) = \biglor \Set{ y | f^*(y) \leq x}\]
        As it is monotone and preserves finite meets, $f_*$ is an implicative morphism $\O(X) \to \O(Y)$; in particular, it is clearly right adjoint to $f^*$ in $\ArrAlg$, which is then computationally dense.
\end{proof}

The converse is also true: computationally dense implicative morphisms between frames are themselves frame homomorphisms. Indeed, let $f : \O(Y) \to \O(X)$ be a computationally dense implicative morphism between frames and let $h : \O(X) \to \O(Y)$ be right adjoint to it. As above, we know that $f$ and $h$ preserve finite meets, so in particular they are monotone; moreover, as the logical and the evidential order coincide, $h$ is right adjoint to $f$ as monotone maps between the posets underlying $\O(Y)$ and $\O(X)$, which means that $f$ preserves all joins and hence that it is a frame homomorphism. Summing up, we have shown the following.

\begin{proposition}
The canonical inclusion of frames into arrow algebras determines a 2-fully faithful 2-functor $\Frm \inc \ArrAlgcd$.

Explicitly, this means that for all frames $\O(Y)$ and $\O(X)$ there is an equivalence of preorder categories:
\[\Frm(\O(Y), \O(X) ) \equiv \ArrAlgcd(\O(Y),\O(X))\]
\end{proposition}

\begin{remark}
In essence, this makes so that the canonical embedding of locales and their homomorphisms into localic triposes and geometric morphisms factors through (the opposite category of) arrow algebras and computationally dense implicative morphisms. In the `algebraic' notation we have been using throughout the paper, this gives the following diagram:
\[\begin{tikzcd}
	\Frm && {\Tripgeom{\Sets}} \\
	& \ArrAlgcd
	\arrow[hook, from=1-1, to=1-3]
	\arrow[hook, from=1-1, to=2-2]
	\arrow[hook, from=2-2, to=1-3]
\end{tikzcd}\]
\end{remark}

\subsection{Partial combinatory algebras}Let us start by summarizing the theory of transformations of realizability triposes coming from morphisms between PCAs, once again following \cite{zoethout2022}. 

First, let us start by linking morphisms of PCAs with cartesian transformations of realizability triposes. Any transformation of realizability triposes $\Phi : P_{\AA} \to P_{\BB}$ is given up to isomorphism by postcomposition with the function $f \coloneqq \Phi_{DA}(\id_{DA}): D A \to D B$ at each component: as shown in \cite[Prop. 3.3.16]{zoethout2022}, $\Phi$ is cartesian if and only if $f$ is a morphism of PCAs, and the respective orders agree as well. In other words, we have the following.

\begin{proposition}\label{lextrans of realtrip}
The association $f \mapsto f \circ -$ is 2-functorial on downsets PCAs and, for all PCAs $\AA$ and $\BB$, it realizes an equivalence of preorder categories:
    \[\Hom\oPCA {D\AA} {D\BB} \equiv \Hom{\Tripcart\Sets}{P_{\AA}}{P_{\BB}}\]
\end{proposition}

Instead, partial applicative morphisms $\AA \to \BB$ are characterized as those inducing \emph{regular} transformations of triposes, which we now introduce.

\begin{definition}
 A transformation of triposes $\Phi^+ : P \to Q$ is \emph{regular} if it is cartesian and preserves existential quantification, i.e. if:
    \[(\Phi^+)_Y \circ \exists_g \liso_Y \exists_g \circ (\Phi^+)_X\]
    for every function $g : X \to Y$.
    
    We denote with $\Tripreg\Sets$ the wide subcategory of $\Tripcart \Sets$ on regular transformations.
\end{definition}
\begin{remark}
Regular transformations preserve the interpretation of \emph{regular logic}, the fragment of finitary first-order logic defined by $\top$, $\land$ and $\exists$.
\end{remark}

Consider now a partial applicative morphism $f : \AA \to \BB$, that is, a morphism of PCAs $f : \AA \to D \BB$. Recall that $f$ corresponds essentially uniquely to the $D$-algebra morphism:
\[\tilde{f} : D\AA \to D \BB\qquad \tilde{f}(\alpha) \coloneqq \bigcup_{a\in \alpha} f(a)\]
and the 2-functor defined by $f \mapsto \tilde{f}$ realizes an equivalence of preorder categories between partial applicative morphisms $\AA \to \BB$ and $D$-algebra morphisms $D\AA \to D \BB$. 

Therefore, the correspondence stated in the previous proposition restricts to partial applicative morphisms and regular transformations: indeed, a cartesian transformation $g \circ - : P_{\AA} \to P_{\BB}$ is regular if and only if $g : D \AA \to D \BB$ is a $D$-algebra morphism, i.e. if and only if it is up to isomorphism of the form $g = \tilde{f}$ for an essentially unique partial applicative morphism $f : \AA \to \BB$. In other words, we have the following.

\begin{proposition}\label{regtrans of realtrip}
The association $f \mapsto \tilde{f} \circ -$ determines a 2-fully faithful 2-functor $\oPCA_D \inc \Tripreg\Sets$.

Explicitly, this means that for all PCAs $\AA$ and $\BB$ there is an equivalence of preorder categories:
    \[\Hom{\oPCA_D} {\AA} {\BB} \equiv \Hom{\Tripreg\Sets}{P_{\AA}}{P_{\BB}}\]
\end{proposition}

Finally, we can specify the previous correspondence to geometric morphisms by means of computational density. First, note how computational density can be characterized by the existence of right adjoints in $\oPCA$.

\begin{lemma}[\protect{\cite[Cor. 2.3.15]{zoethout2022}}]\label{zoe-2.3.15}
    Let $f : \AA\to \BB$ be a partial applicative morphism. Then, $f$ is computationally dense if and only if $\tilde{f}: D \AA \to D\BB$ has a right adjoint in $\oPCA$.
\end{lemma}

As the existence of right adjoints in $\oPCA$ precisely corresponds to the existence of right adjoints on the level of transformations of triposes, this yields the following. 

\begin{theorem}\label{geo of realtrip}
    Let $\oPCA_{D,\mathsf{cd}}$ be the wide sub(-bi)category of $\oPCA_D$ on computationally dense partial applicative morphisms. 

    The association $f \mapsto \tilde{f} \circ -$ determines a 2-fully faithful 2-functor ${\oPCA_{D,\mathsf{cd}} \inc \Tripgeom\Sets}$.

    Explicitly, this means that for all PCAs $\AA$ and $\BB$ there is an equivalence of preorder categories:
    \[\Hom{\oPCA_{D,\mathsf{cd}}} {\AA} {\BB} \equiv \Hom{\Tripgeom\Sets}{P_{\AA}}{P_{\BB}}\]
    In particular, a right adjoint of $\tilde{f} \circ - : P_{\AA} \to P_{\BB}$ is given by $h \circ - : P_{\BB} \to P_{\AA}$, where $h : D \BB \to D \AA$ is right adjoint to $\tilde{f}$ in $\oPCA$.
\end{theorem}

Let us now see how arrow algebras fit in the picture.

First, recall by \autoref{lem:applicative implicative} that any morphism of PCAs $D \AA \to D \BB$, given two PCAs $\AA = (A,\leq, \cdot, A^{\#})$ and $\BB = (B, \leq, \cdot, B^{\#})$, is an implicative morphism between the associated arrow algebras. \autoref{prop:lex trans is induced by morphism} now allows us to easily address the converse.

\begin{proposition}
Let $f : D \AA \to D \BB$ be an implicative morphism. Then, $f$ is also a morphism of PCAs $D \AA \to D \BB$.

Therefore:
\[\Hom\oPCA{D\AA}{D\BB} = \Hom\ArrAlg{D\AA}{D\BB} \]
\end{proposition}
\begin{proof}
Indeed, $f$ induces by postcomposition the cartesian transformation of realizability triposes $\Phi_f^+ : P_{D\AA} \to P_{D\BB}$; by \autoref{lextrans of realtrip}, therefore, $f$ is a morphism of PCAs $D \AA \to D \BB$. Recalling that the two orders coincide as well, we conclude that $\Hom\oPCA{D\AA}{D\BB} = \Hom\ArrAlg{D\AA}{D\BB}$.
\end{proof}

Moving on to partial applicative morphisms $\AA \to \BB$, recall that they correspond to regular transformations of triposes. The following definition is then obvious.

\begin{definition}\label{def:regular morphism}
Let $\A$ and $\B$ be arrow algebras. An implicative morphism $f : \A \to \B$ is \emph{regular} if, for every function $g : X \to Y$ and every $\alpha \in P_{\A}(X)$:
    \[f \circ \exists_g(\alpha) \liso_Y \exists_g (f\circ \alpha)\]

We denote with $\ArrAlgreg$ the wide subcategory of $\ArrAlg$ on regular implicative morphisms; the 2-functor of \autoref{prop:lex trans is induced by morphism} obviously restricts to a 2-fully faithful 2-functor $\ArrAlgreg \inc  \Tripreg \Sets$.
\end{definition}

\begin{remark}
Note that the inequality $\exists_g (f\alpha) \vdash_Y f\exists_g(\alpha)$ holds for every implicative morphism $f$: indeed, through the adjunction $\exists_g \dashv g^*$ it is equivalent to $f \alpha \vdash_X f \exists_g (\alpha) g$, which is ensured by the properties of $f$ since $\alpha \vdash_X \exists_g(\alpha) g$ by the unit of the same adjunction.  Therefore, regularity amounts to the inequality $f \exists_g (\alpha) \vdash_Y \exists_g(f \alpha)$.
\end{remark}

\begin{remark}
    Computationally dense implicative morphism are regular, since geometric transformations of triposes are regular. 
\end{remark}

Drawing from the previous results, we conclude that regular implicative morphisms between arrow algebras arising from PCAs arise themselves from partial applicative morphisms.

\begin{proposition}
    The 2-functor $\tilde{D}$ of \autoref{prop:tildeD} restricts to a 2-fully faithful 2-functor ${\oPCA_D \inc \ArrAlgreg}$.

    Explicitly, this means that for all PCAs $\AA$ and $\BB$, $\tilde{D}$ realizes an equivalence of preorder categories:
    \[\Hom{\oPCA_D}{\AA}{\BB} \iso \Hom\ArrAlgreg {D \AA}{D\BB}\]
\end{proposition}
\begin{proof}
Let $f : D \AA \to D\BB$ be a regular implicative morphism. Then, $f$ induces by postcomposition the regular transformation of realizability triposes $\Phi_f^+ : P_{D\AA} \to P_{D\BB}$; by \autoref{regtrans of realtrip}, therefore, $f=\tilde{D}g$ for an essentially unique partial applicative morphism $g : \AA \to \BB$\footnote{We can also describe $g$ as $f\circ \delta'_{\AA}$.}. 

        Moreover, for $f,f' : \AA \to \BB$ partial applicative morphisms, we have already shown that $f \leq f'$ in $\Hom{\oPCA_D}{\AA}{\BB}$ if and only if $\tilde{D} f \vdash \tilde{D} f'$ in $\Hom \ArrAlg{D\AA}{D\BB}$.
\end{proof}

\begin{remark}
    Alternatively, the proof of the previous can be given by observing that a regular implicative morphism $f : D \AA \to D \BB$ is a union-preserving morphism of PCAs, and hence a $D$-algebra morphism: in fact, $D\AA$ and $D \BB$ are compatible with joins, meaning that existentials can be computed as unions (cfr. \cite[Lem. 5.3]{berg2023arrow}). 
\end{remark}

Finally, let us specialize to the case of computational density. Recall by \autoref{zoe-2.3.15} that a partial applicative morphism $f : \AA \to \BB$ is computationally dense if and only if $\tilde{f} : D \AA \to D \BB$ has a right adjoint in $\oPCA$. Since $\Hom\oPCA{D\AA}{D\BB}$ coincides as a preorder with $\Hom\ArrAlg{D\AA}{D\BB}$, this is also equivalent to $\tilde{f} : D\AA \to D \BB$ having a right adjoint in $\ArrAlg$, that is, to $\tilde{f}$ being computationally dense as an implicative morphism $D \AA \to D \BB$. In essence, we have shown the following.

\begin{proposition}
    The 2-functor $\tilde{D}$ of \autoref{prop:tildeD} restricts to a 2-fully faithful 2-functor $\oPCA_{D,\mathsf{cd}} \inc  \ArrAlgcd$.

Explicitly, this means that for all PCAs $\AA$ and $\BB$, $\tilde{D}$ realizes an equivalence of preorder categories:
    \[{ \oPCA_{D,\mathsf{cd}} } (\AA, \BB) \equiv \ArrAlgcd(D\AA, D\BB)\]
\end{proposition}

\begin{remark}
    As for frames, in essence this makes so that the construction of realizability triposes and geometric morphisms from PCAs and partial applicative morphisms factors through arrow algebras and computationally dense implicative morphisms, giving the following diagram:
    \[\begin{tikzcd}
	{ \oPCA_{D,\mathsf{cd}} } && {\Tripgeom{\Sets}} \\
	& \ArrAlgcd
	\arrow[hook, from=1-1, to=1-3]
	\arrow[hook, from=1-1, to=2-2]
	\arrow[hook, from=2-2, to=1-3]
\end{tikzcd}\]
\end{remark}

\section{Inclusions of arrow triposes}\label{sec:incl}
In this section, we will specify the previous correspondence between computationally dense implicative morphisms and geometric morphisms of arrow triposes to the case of geometric inclusions into a given arrow tripos $P_{\A}$, and see how they correspond to nuclei on $\A$.

\subsection{Inclusions and surjections}Recall that a geometric morphism of arrow triposes $\Phi : P_{\B} \to P_{\A}$ is an \emph{inclusion} if $(\Phi_+)_I$ reflects the order for every set $I$, or equivalently if $(\Phi^+)_I(\Phi_+)_I(\phi) \liso_I \phi$ for every set $I$ and every $\phi : I \to B$. Dually, $\Phi$ is a \emph{surjection} if $(\Phi^+)_I$ reflects the order for every set $I$, or equivalently if $(\Phi_+)_I(\Phi^+)_I(\phi) \liso_I \phi$ for every set $I$ and every $\phi : I \to A$.

Recall moreover that, in any preorder-enriched category $\C$:
\begin{itemize}
    \item[--]an arrow $f : A \to B$ is a \emph{lax epimorphism} if, for every $C \in \C$, the map
\[-\circ f : \C(B, C) \to \C(A, C)\]
is fully-faithful as a functor between preorder categories, which explicitly means that $p \leq q$ for all $p, q : B \to C$ such that $pf \leq qf$;
\item[--]an arrow $f : A \to B$ is a \emph{lax monomorphism} if, for every $C \in\C$, the map
\[f\circ - : \C(C, A) \to \C(C, B)\]
is fully-faithful as a functor between preorder categories, which explicitly means that $p \leq q$ for all $p, q : C \to A$ such that $fp \leq fq$.
\end{itemize}

Specializing to $\ArrAlg$, we can then give the following definition.

\begin{definition}
    A computationally dense implicative morphism $f : \A \to \B$ is an \emph{implicative surjection} (resp. \emph{implicative injection}) if it is a lax epimorphism (resp. lax monomorphism) in $\ArrAlg$.
\end{definition}

\begin{proposition}\label{prop:inclusions}
    Let $f : \A \to \B$ be a computationally dense implicative morphism with right adjoint $h : \B \to \A$ and let $\Phi : P_{\B} \to P_{\A}$ be the induced geometric morphism of arrow triposes. The following are equivalent:
    \begin{enumerate}
        \item $\Phi$ is an inclusion;
        \item $f h \liso_B \id_B$;
        \item $f$ is an implicative surjection.
\end{enumerate}

Dually, the following are equivalent:
    \begin{enumerate}
        \item $\Phi$ is a surjection;
        \item $h f \liso_A \id_A$;
        \item $f$ is an implicative injection. 
    \end{enumerate}
\end{proposition}
\begin{proof}
For (1) $\Leftrightarrow$ (2), recall that the inverse image $\Phi^+$ is given by postcomposition with $f$, and the direct image $\Phi_+$ is given by postcomposition with $h$: therefore, $\Phi$ is an inclusion if and only if $f h \phi \liso_I \phi$ for every set $I$ and every $\phi \in P_{\B}(I)$, which is equivalent to $f h \liso_B \id_B$.

        For (2) $\implies$ (3), suppose $p,q : \B \to \mathcal{C}$ are such that $pf \vdash qf$. Then, $pfh \vdash qfh$, and hence $p \vdash q$. 

        For (3) $\implies$ (2), of course $fh \vdash \id_B$; conversely, to show that $\id_B \vdash fh$ it then suffices to show that $f \vdash fhf$, which is ensured by $\id_A \vdash hf$. 
\end{proof}

\begin{corollary}
For all arrow algebras $\A$ and $\B$, there are equivalences of preorder categories between:
\begin{itemize}
    \item[--] implicative surjections $\A \to \B$ and geometric inclusions $P_{\B} \inc P_{\A}$;
    \item[--] implicative injections $\A \to \B$ and geometric surjections $P_{\B} \epi P_{\A}$.
\end{itemize}
\end{corollary}
\begin{proof}    Combining \autoref{prop:geo is induced by cd morphism} with the previous proposition.
\end{proof}

\begin{remark}
    Of course, an implicative morphism is an equivalence if and only if it is both an implicative surjection and an implicative inclusion.
\end{remark}

\subsection{Nuclei and subtriposes}Let $\A = (A,\evleq,\arr,S)$ be an arrow algebra.

Recall from \cite[Def. 3.11]{berg2023arrow} that a \emph{nucleus} on $\A$ is a function $j : A \to A$ such that:
\begin{enumerate}[label=\roman*.]
    \item if $a \evleq b$ then $ja \evleq jb$;
    \item $\evmeet_{a \in A} a \arr ja \in S$;
    \item $\evmeet_{a,b \in A} (a\arr jb)\arr ja \arr jb \in S$.
\end{enumerate}
which also imply:
\begin{enumerate}[label=\roman*.]
    \item[iv.] $\evmeet_{a\in A} jja \arr ja \in S$;
    \item[v.] $\evmeet_{a,b \in A} (a\arr b) \arr ja \arr jb \in S$;
    \item[vi.] $\evmeet_{a,b \in A} j(a\arr b)\arr ja \arr jb \in S$,
\end{enumerate}	
and we can even substitute (iii) in the definition with the conjunction of (iv) and (vi). Every nucleus $j$ on $\A$ determines the new arrow algebra $\mathcal{A}_j = (A, \evleq, \arr_j, S_j)$, where:
	\[	a \arr_j b \coloneqq a \arr jb \qquad S_j \coloneqq \Set{ a \in A | ja \in S}\]
We denote with $\vdash^j$ the logical order in $\A_j$: explicitly, $a \vdash^j b$ if and only if $j ( a \arr j b ) \in S$, which by the properties of nuclei and separators is equivalent to $a \arr j b \in S$ and hence to $a \vdash jb$. Note also that $S \subseteq S_j$: in fact, if $a \in S$, then since $\evmeet_{a \in A} a \arr ja \in S$ by modus ponens it follows that $ja \in S$, which precisely means $a\in S_j$.

With the machinery of the previous sections, \cite[Prop. 6.3]{berg2023arrow} can then be reduced to the following observation.

\begin{lemma}\label{lem:nuclei are right adjoints}
    $\id_A$ is an implicative surjection $\mathcal{A}\to \mathcal{A}_j$, with $j$ as a right adjoint.
\end{lemma}
\begin{proof}
Let us start by showing that $\id_A$ is an implicative morphism $\mathcal{A}\to \mathcal{A}_j$.  As the evidential order is the same in $\A$ and $\A_j$, we only have to verify (i) and (ii) in \autoref{def:implicative morphism of aas}.
        \begin{enumerate}[label=\roman*.]
            \item If $a \in S$, then $ja \in S$, meaning that $a \in S_j$.

            \item Condition (ii) explicitly reads as:
            \[\evmeet_{a,a'}  ( a \arr a' ) \arr_j a \arr_j a' \in S_j\]
            i.e.:
            \[j \left(\evmeet_{a,a'} ( a \arr a' ) \arr j ( a \arr j a')\right) \in S \]
            so, by (ii) in the definition of a nucleus, it suffices to show:
            \[\evmeet_{a,a'} ( a \arr a' ) \arr j ( a \arr j a') \in S \]
            This, in turn, follows by intuitionistic reasoning from:
            \begin{gather*}
            \evmeet_{a ,a'} (a\arr a')\arr a \arr ja' \in S \\
            \evmeet_{a,a'}(a\arr ja')\arr j(a\arr ja')\in S 
            \end{gather*}
            both of which follow again from (ii).
        \end{enumerate}

Then, let us show that $j$ is an implicative morphism $\A_j\to\A$: again, recall that $j$ is monotone by definition, so we only have to verify (i) and (ii) in \autoref{def:implicative morphism of aas}.
\begin{enumerate}[label=\roman*.]
\item If $a \in S_j$, then by definition $j a \in S$.
\item Condition (ii) explicitly reads as:
\[\evmeet_{a,a'} j ( a \arr j a') \arr j a \arr j a' \in S\]
which follows from intuitionistic reasoning from:
\[\evmeet_{a,a'} j (a \arr ja') \arr ja \arr jja' \in S \qquad  \evmeet_{a'} jja' \arr ja' \in S\]
\end{enumerate}

Finally, let us show that $j : \A_j\to\A$ is right adjoint to $\id_A : \A \to \A_j$ in $\ArrAlg$.
\begin{itemize}
    \item[--] On one hand, $j \vdash^j_A \id_A$ explicitly reads as $j \vdash_A j$, which is clearly true.
    \item[--] On the other, $\id_A \vdash_A j$ is true as $j$ is a nucleus.
\end{itemize}
Moreover, we also have that $\id_A \vdash_A^j j$ as it explicitly reads as $\id_A \vdash_A j j$, which makes $\id_A$ an implicative surjection by \autoref{prop:inclusions}.
\end{proof}

\begin{corollary}
    Every nucleus $j$ on $\A$ induces a geometric inclusion of triposes $P_{\A_j}\inc P_{\A}$, given by:
     \[\begin{tikzcd}
	{P_{\A_j}} & {P_{\A}}
	\arrow[""{name=0, anchor=center, inner sep=0}, "{{\id_A} \circ - }"', curve={height=12pt}, from=1-2, to=1-1]
	\arrow[""{name=1, anchor=center, inner sep=0}, "{j {\circ} {-} }"', curve={height=12pt}, from=1-1, to=1-2]
	\arrow["\dashv"{anchor=center, rotate=-90}, draw=none, from=0, to=1]
\end{tikzcd} \]
\end{corollary}

However, we are now in the position to do more than that: namely, we can recover and extend \autoref{standard inclusion} to a correspondence between subtoposes of an arrow topos and nuclei on the underlying arrow algebra, hence proving the converse of the previous.

Recall in fact by the discussion in \autoref{sec:triposes} that we have an equivalence of preorder categories between subtriposes of $P_{\A}$ and closure transformations on $P_{\A}$, that is, transformations $\Phi_j : P_{\A} \to P_{\A}$ which are cartesian, inflationary and idempotent:
\[\SubTrip{ P_{\A}} \iso \op{\ClTrans {P_{\A}}}\]
By the correspondence in \autoref{prop:lex trans is induced by morphism}, a transformation $\Phi_j : P_{\A} \to P_{\A}$ is a closure transformation on $P_{\A}$, if and only if the function $j \coloneqq (\Phi_j)_A(\id_A) : A \to A$ inducing it up to isomorphism satisfies the following:
\begin{enumerate}[label=\roman*.]
    \item $j$ is an implicative morphism $\A \to \A$;
    \item $\id_A \vdash j$;
    \item $j j \liso j$.
\end{enumerate}
Assuming up to isomorphism $j$ to be monotone with respect to the evidential order in $\A$ as in \autoref{lem:monotonicity}, note then how this is equivalent to $j$ satisfying (i), (ii), (iv) and (vi) in the definition of a nucleus, which as we've noted is equivalent to asking that $j$ is a nucleus.  Since the association $j \mapsto \Phi_j$ also preserves and reflects the order, we conclude with the following.

\begin{proposition}\label{prop:nuclei and subtriposes}
    Let $\N \A$ be the set of nuclei on $\A$, with the preorder induced by $P_{\A}(A)$. Then, \autoref{prop:lex trans is induced by morphism} yields an equivalence of preorder categories:
    \[\ClTrans {P_{\A}} \iso \N \A \]
    so, in particular:
    \[ \SubTrip{P_{\A}} \iso \op {\N \A }\]
\end{proposition}

\begin{corollary}
Every geometric inclusion of toposes into $\AT \A$ is induced, up to equivalence, by a geometric inclusion of triposes of the form:
   \[\begin{tikzcd}
	{P_{\A_j}} & {P_{\A}}
	\arrow[""{name=0, anchor=center, inner sep=0}, "{{\id_A} \circ - }"', curve={height=12pt}, from=1-2, to=1-1]
	\arrow[""{name=1, anchor=center, inner sep=0}, "{j {\circ} {-} }"', curve={height=12pt}, from=1-1, to=1-2]
	\arrow["\dashv"{anchor=center, rotate=-90}, draw=none, from=0, to=1]
\end{tikzcd} \]
for some nucleus $j$ on $\A$.
\end{corollary}

\begin{remark}
    By definition of $\A_j$, note therefore how $P_{\A_j}$ coincides precisely with the tripos $P_j$ described before \autoref{standard inclusion}. For this reason, we will usually refer to the subtripos $P_{\A_j} \inc P_{\A}$ simply as $P_j\inc P_{\A}$.
\end{remark}

We conclude this part with the following alternative description of $P_j$, already noted in the general case in \cite{van2008realizability} and then in the context of arrow algebras in \cite[Prop. 6.6]{berg2023arrow}. We record it here to have an explicit description of the corresponding geometric inclusion, which will come back in \autoref{sec:mod-real}.

\begin{proposition}\label{prop:alternative subtripos}
    Let $j\in \N \A$. Then, $P_j \in \SubTrip{P_{\A}}$ is equivalent to the subtripos $Q_j \inc P_{\A}$ defined by:
    \[ Q_j (I) \coloneqq \Set{ \alpha \in P_{\A}(I) | j \alpha \vdash_I \alpha }\]
    with the Heyting prealgebra structure induced by $P_{\A}(I)$.
\end{proposition}
\begin{proof}
Consider the pair of transformations:
\[\Theta^+ \coloneqq {\id_A} \circ - : Q_j \to P_j \qquad \Theta_+ \coloneqq j\circ- :P_j\to Q_j\]
obviously well-defined since so is the geometric morphism $P_j \inc P_{\A}$ above, and as $j j \vdash_A j $. Then, $\Theta^+$ and $\Theta_+$ define an equivalence of triposes between $P_j$ and $Q_j$.
\begin{itemize}
    \item[--] The fact that $\Theta^+\Theta_+ \iso \id_{P_j}$ is equivalent, by \autoref{prop:lex trans is induced by morphism}, to $j \liso^j_A \id_A$: on one hand, $\id_A \vdash^j_A j $ explicitly means $\id_A \vdash_A j j$, which follows from $\id_A \vdash_A j$; on the other, $j \vdash^j_A \id_A$ explicitly means $j \vdash_A j$, which follows by reflexivity.

    \item[--] To show that $\Theta_+\Theta^+ \iso \id_{Q_j}$ we need to show that $j \alpha \liso_I \alpha$ for every set $I$ and every $\alpha \in Q_j(I)$: on one hand, $\alpha \vdash_I j \alpha$ follows by $\id_A \vdash_A j$; on the other, $j \alpha \vdash_I \alpha$ follows by definition of $Q_j(I)$.
\end{itemize}
Therefore, $Q_j$ is equivalent to $P_j$; through the equivalence, the geometric inclusion of $Q_j$ in $P_{\A}$ is given by:
     \[\begin{tikzcd}
	{Q_j} & {P_{\A}}
	\arrow[""{name=0, anchor=center, inner sep=0}, "{j {\circ} {-} }"', curve={height=12pt}, from=1-2, to=1-1]
	\arrow[""{name=1, anchor=center, inner sep=0}, "{j {\circ} {-} }"', curve={height=12pt}, from=1-1, to=1-2]
	\arrow["\dashv"{anchor=center, rotate=-90}, draw=none, from=0, to=1]
\end{tikzcd} \]
\end{proof}

\subsection{A factorization theorem}As it is known, every geometric morphism of toposes can be factored as a geometric surjection followed by a geometric inclusion. Generalizing locale theory, let us recover the same result on the level of arrow algebras; in doing so, we will also make the correspondence between subtriposes and nuclei more explicit.

Let $f : \A\to \B$ be a computationally dense implicative morphism with right adjoint $h : \B \to \A$; by \autoref{lem:monotonicity}, up to isomorphism we can assume both $f$ and $h$ to be monotone with respect to the evidential order. First, observe the following.

\begin{lemma}
    $hf : \A \to \A$ is a nucleus on $\A$.
\end{lemma}
\begin{proof}    Let us verify that $hf$ satisfies the three conditions defining nuclei.
    \begin{enumerate}[label=\roman*.] 
    \item $hf$ is monotone by composition.
    \item Clearly $\id_A \vdash hf$ as $h$ is right adjoint to $f$.
    \item Let $I \coloneqq A \times A$; note that condition (iii) can be rewritten as:
    \[ \pi_1 \arr hf\pi_2 \vdash_I hf\pi_1 \arr hf \pi_2 \]
    where $\pi_1,\pi_2 : I \to A$ are the obvious projections. Through the Heyting adjunction in $P_{\A}(I)$, this is equivalent to:
    \[ (\pi_1 \arr hf\pi_2 )\land hf\pi_1 \vdash_I hf \pi_2 \]
    and hence, since $h\circ -$ is right adjoint to $f \circ -$, which preserves finite meets, to:
    \[ f(\pi_1 \arr hf\pi_2) \land fhf\pi_1 \vdash_I f \pi_2 \]
    Therefore, since $fh f\pi_1 \vdash_I f\pi_1$ as $fh \vdash \id_B$, it suffices to show:
    \[ f(\pi_1 \arr hf\pi_2) \land f\pi_1 \vdash_I f \pi_2 \]
    i.e., again through the Heyting adjunction in $P_{\B}(I)$:
    \[ f(\pi_1 \arr hf\pi_2) \vdash_I f\pi_1 \arr f\pi_2 \]
    which in turn follows since, by (ii) in \autoref{def:implicative morphism of aas}:
    \[f (\pi_1\arr hf \pi_2) \vdash_I f\pi_1 \arr fhf\pi_2 \]
    and again $fhf\pi_2 \vdash_I f\pi_2$.
\end{enumerate} 
\end{proof}

A natural question is then to relate $f$ to $j\coloneqq hf$, and in particular the geometric morphism $\Phi_f : P_{\B} \to P_{\A}$ induced by $f$ with the inclusion $\Phi_j : P_{j} \inc P_{\A}$ induced by $j$. To this aim, recall here that $f h \vdash \id_B$ and $\id_A \vdash hf$ imply $fhf \liso f$ and $hfh \liso h$.

\begin{lemma}
    $f$ is an implicative injection $\A_j \to \B$, with $h$ as a right adjoint.
\end{lemma}
\begin{proof}
Let us start by showing that $f$ is an implicative morphism $\A_j \to \B$. 
\begin{enumerate}[label=\roman*.]
        \item Let $a \in S_j$; by definition, $hf(a) \in S_A$, so $fhf(a) \in S_B$, and hence $f(a) \in S_B$ since $fh \vdash \id_B$.
        \item Condition (ii) explicitly reads as:
        \[\evmeet_{a,a'} f(a\arr ja') \arr f(a) \arr f(a') \in S_B\]
        which follows by intuitionistic reasoning from:
        \begin{gather*}
            \evmeet_{a,a'} f(a\arr ja') \arr f(a) \arr fj(a') \in S_B \\
            \evmeet_{a,a'} (f(a) \arr fhf(a')) \arr f(a) \arr f(a') \in S_B
        \end{gather*}
        where the latter follows since $fhf \vdash f$.
        \end{enumerate}

Note now that $h : B \to A$ is an implicative morphism $\B \to \A_j$ since it is an implicative morphism $\B \to \A$ and $\id_A$ is an implicative morphism $\A \to \A_j$. Then, we have that $h : \B \to \A_j$ is right adjoint to $f: \A_j \to \B$:
\begin{itemize}
    \item[--]clearly $f h \vdash \id_B$;
    \item[--] on the other hand, $\id_A \vdash^j h f$ explicitly reads as $\id_A \vdash_A h f h f $, which follows from $\id_A \vdash_A h f$.
\end{itemize}

Moreover, we also have that $hf \vdash_A^j \id_A$ as it explicitly reads as $hf \vdash_A hf$, which makes $f : \A_j \to \B$ an implicative surjection by \autoref{prop:inclusions}.
\end{proof}

Recalling by \autoref{lem:nuclei are right adjoints} that $\id_A$ defines an implicative surjection $\A \to \A_j$, we have the following.

\begin{corollary}
    Every computationally dense implicative morphism factors as an implicative surjection followed by an implicative inclusion.
\end{corollary}

On the level of triposes, this means that the geometric morphism $\Phi_f : P_{\B} \to P_{\A}$ induced by $f$ factors through $\Phi_j : P_j \inc P_{\A}$ by means of a geometric surjection $\Theta_f : P_{\B} \to P_j$, also induced by $f$ as a morphism $\A_j \to \B$:
\[\begin{tikzcd}
	{P_{\B}} \\
	&& {P_{\A}} \\
	{P_j}
	\arrow[""{name=0, anchor=center, inner sep=0}, "{f\circ-}"{description}, curve={height=12pt}, from=2-3, to=1-1]
	\arrow[""{name=1, anchor=center, inner sep=0}, "{h\circ-}"{description}, curve={height=12pt}, from=1-1, to=2-3]
	\arrow[""{name=2, anchor=center, inner sep=0}, "{{\id_A}\circ-}"{description}, curve={height=12pt}, from=2-3, to=3-1]
	\arrow[""{name=3, anchor=center, inner sep=0}, "{j\circ-}"{description}, curve={height=12pt}, from=3-1, to=2-3]
	\arrow["\Theta_f"', from=1-1, to=3-1]
	\arrow["\dashv"{anchor=center, rotate=-114}, draw=none, from=0, to=1]
	\arrow["\dashv"{anchor=center, rotate=-66}, draw=none, from=2, to=3]
\end{tikzcd}\]

\begin{proposition}
    $\Theta_f$ is an equivalence if and only if $\Phi_f$ is an inclusion.
\end{proposition}
\begin{proof}
By \autoref{lem:equivalences} and \autoref{prop:equivalent triposes are induced by equivalences}, $\Theta_f$ is an equivalence if and only if $f : \A_j \to \B$ is an equivalence. Since $h : \B \to \A_j$ is right adjoint to $f: \A_j \to \B$, this is equivalent to $hf \vdash^j_A \id_A$ and $\id_B \vdash fh$, and since $hf \vdash^j_A \id_A$ holds trivially this is equivalent simply to $\id_B \vdash fh$. By \autoref{prop:inclusions}, $\id_B \vdash fh$ is in turn equivalent to $\Phi_f$ being an inclusion.
\end{proof}

\begin{remark}\label{rem:nucleus induced by surjection}
    In essence, this gives us a more explicit description of the correspondence given in \autoref{prop:nuclei and subtriposes}, in perfect generalization of the localic case: indeed, if a subtripos $\Phi : P_{\B} \inc P_{\A}$ is induced by an implicative surjection $f : \A \to \B$, then it is equivalent to the subtripos induced by (a nucleus isomorphic to) $hf$, where $h$ is right adjoint to $f$. 
\end{remark}

\section{Arrow algebras for modified realizability}\label{sec:mod-real}
In this section, we will apply the theoretical framework developed above to lift the study of \emph{modified realizability} to the level of arrow algebras. Modified realizability is a variant of Kleene's number realizability introduced by Kreisel in 1959 \cite{Kreisel1959-KREIOA}. On the semantical side, a topos for Kreisel's modified realizability was first defined by Greyson in 1981 \cite{graymod}; see also \cite{vanoosten1997, Birkedal2002-BIRRAM-2, johnstone2017}. 

The key feature of modified realizability lies in separating between a set of \emph{potential realizers} and a subset thereof of \emph{actual realizers}. On the level of triposes, this amounts to shifting from the ordinary realizability tripos $P_{\AA}$ over a (traditionally, discrete and absolute) PCA $\AA$ to a tripos whose predicates on a set $I$ are functions from $I$ to the set:
\[\Set{ (\alpha,\beta) \in D A \times D A | \alpha \subseteq \beta }\]
which are preordered by:
\[ \phi \vdash_I \psi \iff \bigcap_{i\in I} (\phi_1(i) \arr \psi_1(i)) \cap (\phi_2(i) \arr \psi_2(i)) \in (DA)^{\#}  \]
where we denote with $\phi_1(i),\phi_2(i)$ the two components of $\phi(i)$. This idea is what led, in \cite{berg2023arrow}, to the definition of the \emph{Sierpiński construction} on arrow algebras, which we will describe below. To do this, and for what follows, we will consider some \emph{ad hoc} notions of arrow algebras which still encompass all the relevant cases and many others. This does not mean that we have counterexamples showing how the presented results may fail for more general classes of arrow algebras, but only that minimal assumptions suffice to ensure the desired properties.

\begin{definition}
    An arrow algebra $\A = (A,\evleq,\arr,S)$ is \emph{binary implicative} if the equality:
\[\textstyle a \arr (b \evmeet c) = a \arr b \evmeet a \arr c\]
holds for all $a,b,c\in A$, and it is \emph{modifiable} if moreover the equality:
\[\bot \arr a = \top\]
holds for every $a \in A$.

We denote with $\ArrAlgbi$ and $\ArrAlgmod$ the full subcategories of $\ArrAlg$ on binary implicative and modifiable arrow algebras, respectively.
\end{definition}

\begin{example}
    Every frame, seen as an arrow algebra in the canonical way, is modifiable.
\end{example}
\begin{example}
    For every PCA $\mathbb{A}$, $D \AA$ is modifiable; in particular, $\operatorname{PER} \AA$ (cfr. \cite[Thm. 3.10]{berg2023arrow}) is modifiable.
\end{example}

\subsection{The Sierpiński construction}\label{sec:sierpinski}

Recall by \cite[Prop. 7.2]{berg2023arrow} that, starting from any binary implicative arrow algebra $\A = (A, \evleq, \arr, S)$, we can define a new arrow algebra $\s\A = (\s A, \evleq, \arr, \s S)$, also binary implicative, by letting:
\[\s A \coloneqq \Set{ x = (x_0,x_1) \in A \times A | x_0 \evleq x_1}\]
with pointwise order, implication:
\[\textstyle x \arr y \coloneqq ( x_0 \arr y_0 \evmeet x_1 \arr y_1 , x_1 \arr y_1) \]
and separator:
\[\s S \coloneqq \Set{ x \in \s A | x_0 \in S}\]

\begin{remark}
    This means that, for every set $I$, the order in $P_{\s\A}(I)$ is given by:
    \[\phi \vdash_I \psi \iff \evmeet_{i\in I} \textstyle \phi_1(i) \arr \psi_1(i) \evmeet \phi_2(i) \arr \phi_2(i) \in S \]
    where we denote with $\phi_1,\phi_2 : I \to A$ the two components of $\phi : I \to \s A$. 
\end{remark}

Let us now lift the association $\A \mapsto \s\A$ to a (pseudo)functor on $\ArrAlgbi$.

Let $f : \A \to \B$ be an implicative morphism in $\ArrAlgbi$, for the moment assumed to be monotone, and define:
\[\s f : \s A \to \s B \qquad \s f (x_0,x_1) \coloneqq (f(x_0), f(x_1) ) \]

\begin{lemma}
    $\s f$ is an implicative morphism $\s\A \to \s\B$.
\end{lemma}
\begin{proof}
First, note that $\s f$ is well-defined as a function $\s A \to \s B$ by monotonicity of $f$, and it is monotone itself with respect to the evidential orders in $\s \A$ and $\s \B$. Let us then verify that $\s f$ satisfies the first two conditions in \autoref{def:implicative morphism of aas}.
        \begin{enumerate}[label=\roman*.]
\item If $x \in \s S_A$, then $x_0 \in S_A$, so $f(x_0) \in S_B$ and hence $\s{f} (x) \in \s S_B$.

\item First note that, for all $x,y \in \s A$:
    \begin{align*} 
        \s{f} ( x \arr y ) &  \textstyle = \s{f} ( x_0 \arr y_0 \evmeet x_1 \arr y_1 , x_1 \arr y_1 ) \\
        &  \textstyle = ( f(x_0\arr y_0 \evmeet x_1\arr y_1) , f(x_1 \arr y_1) ) 
    \end{align*}
    whereas:
    \begin{align*} 
       \s{ f} (x) \arr \s{f} (y)& = (fx_0, fx_1) \arr (fy_0, fy_1) \\
        & = \textstyle ( fx_0 \arr fy_0 \evmeet fx_1 \arr fy_1, fx_1 \arr fy_1 ) 
    \end{align*}

    Therefore, by binary implicativity, a realizer for $\s f$ amounts to an element $r \in S_B$ such that:
    \begin{align*}
        r & \evleq   \textstyle f(x_0\arr y_0 \evmeet x_1 \arr y_1) \arr fx_0\arr fy_0  \\  
        r & \evleq   \textstyle f(x_0\arr y_0 \evmeet x_1 \arr y_1) \arr  fx_1\arr fy_1 \\
       r & \evleq  f(x_1\arr y_1) \arr fx_1 \arr fy_1
    \end{align*}
    for all $x,y \in \s A$, in which case $(r,r) \in \s S_B$ realizes $\s f$. By monotonicity of $f$, note then that it suffices to show that:
    \begin{align*}
        r & \evleq   f(x_0\arr y_0) \arr fx_0\arr fy_0  \\  
       r & \evleq  f(x_1\arr y_1) \arr fx_1 \arr fy_1
    \end{align*}
    for all $x,y \in \s A$, which means that $r$ can be taken to be a realizer for $f$.
        \end{enumerate}
\end{proof}

Therefore, $\s{(-)}$ defines a functorial association on binary implicative arrow algebras and monotone implicative morphisms between them. Note moreover that, given two monotone implicative morphisms $f,f' : \A \to \B$ in $\ArrAlgbi$, if $u \in S_B$ realizes $f \vdash f'$, then $(u,u) \in \s S_B$ clearly realizes $\s f \vdash \s {f'}$, meaning that $\s{(-)}$ is actually 2-functorial. Precomposing with the pseudofunctor $M$ of \autoref{rem:monotonization}, we obtain the following.

\begin{proposition}\label{prop:sierpinski}
    For every implicative morphism $f : \A \to \B$ in $\ArrAlgbi$, let:
    \[\s f : \s \A \to \s \B \qquad \s f (x_0,x_1) \coloneqq \left( \evmeet_{x_0\evleq a} \partial f (a) , \evmeet_{x_1 \evleq a} \partial f(a) \right) \]
    
    Then, $\s{(-)}$ is a pseudofunctor $\ArrAlgbi\to \ArrAlgbi$.

    Moreover, if $f$ is computationally dense with right adjoint $h : \B \to \A$, then $\s f$ is computationally dense as well, and a right adjoint is given by $\s h : \s \B \to \s \A$.
\end{proposition}
\begin{proof}
We only have to show the last part, so let $f$ be computationally dense with right adjoint $h : \B \to \A$ and, up to isomorphism, assume both $f$ and $h$ to be monotone, so that $\s f$ and $\s h$ can be defined as above in the case of monotonicity. Let us show that $\s h$ is right adjoint to $\s f$.
\begin{itemize}
\item[--] To show that $\s f \s h \vdash \id_{\s B}$, note that:
\begin{align*}
    \MoveEqLeft[3] \evmeet_{y\in \s B} \s f \s h (y) \arr y \in \s S_B \\
    \iff{} & \evmeet_{y\in \s B} ( f h (y_0), fh (y_1) ) \arr (y_0,y_1) \in \s S_B \\
    \iff{} & \evmeet_{y\in \s B}  \textstyle f h (y_0) \arr y_0 \evmeet f h (y_1) \arr y_1 \in S_B
\end{align*}
which is ensured by $f h \vdash \id_B$.

\item[--] Similarly, $\id_{\s A} \vdash \s h \s f$ reduces to $\id_A \vdash h f$.

\end{itemize}
\end{proof}

\begin{corollary}\label{cor:functoriality of sierpinski}
    Let $\A$ and $\B$ be binary implicative arrow algebras. 
    
Every geometric morphism $\Phi : P_{\B} \to P_{\A}$ lifts to a geometric morphism ${\s \Phi : P_{\s \B} \to P_{\s\A}}$. 
\end{corollary}

\begin{question}
    For $\A = \Pow {\mathcal{K}_1}$, we have that $\AT {\s \A}$ is the effective topos built on the topos of sheaves over the Sierpiński space, $\Eff_{\cdot \arr\cdot}$ -- that is, the result of the construction of $\Eff$ inside the topos $\Sets^{\cdot \arr\cdot}$.

    Can we develop a theory of arrow algebras over other base toposes, encompassing that of PCAs over other base toposes, so that the same result holds for every (binary implicative) arrow algebra?
\end{question}

\subsection{The modification of an arrow algebra}

Let us now study the relation between $P_{\A}$ and $P_{\s\A}$. First, generalizing what is shown in \cite{johnstone2013} for discrete and absolute PCAs, we can note the following.

\begin{lemma}
    $P_{\A}$ is a subtripos of $P_{\s\A}$.
\end{lemma}
\begin{proof}
Consider the projection:
        \[\pi_1 : \s A \to A \qquad (x_0,x_1) \mapsto x_1\]
        Let us show that $\pi_1$, which is obviously monotone, is an implicative morphism $\s\A \to \A$.
        \begin{enumerate}[label=\roman*.]
            \item If $(x_0,x_1) \in \s S$, then by definition $x_0 \in S$, so $x_1 \in S$ as well since $x_0 \evleq x_1$.
            \item A realizer of $\pi_1$ amounts to an element $r \in S$ such that:
            \[ r \evleq (x_1 \arr y_1) \arr x_1 \arr y_1\]
            for all $x,y \in \s A$, so we can take $r \coloneqq \bf{i}$.
        \end{enumerate}

        Consider now the diagonal map:
        \[\delta : A \to \s A \qquad a \mapsto (a,a)\]
        Let us show that $\delta$, also obviously monotone, is an implicative morphism $\A \to \s \A$.
        \begin{enumerate}[label=\roman*.]
            \item If $a \in S$, then clearly $(a,a) \in S$.
            \item We have:
    \begin{align*}
\MoveEqLeft[3]    \evmeet_{a,a'\in  A } \delta (a\arr a') \arr \delta(a) \arr \delta (a')  \in \s S \\
\iff{} & \evmeet_{a,a'\in  A }  (a\arr a', a \arr a') \arr (a,a) \arr (a',a')  \in \s S  \\
\iff{} & \evmeet_{a,a' \in A} (a\arr a') \arr a \arr a' \in S
\end{align*}
which is ensured by $\bf{i} \in S$.
        \end{enumerate}

        Finally, let us show that $\delta$ is right adjoint to $\pi_1$ in $\ArrAlg$, making it an implicative surjection.
\begin{itemize}
    \item[--]On one hand, $\pi_1 \delta = \id_A$.
    \item[--]On the other, we have:
    \begin{align*}
\MoveEqLeft[3]    \id_{\s A} \vdash_{\s A} \delta \pi_0 \\
\iff{} & \evmeet_{x\in \s A} (x_0,x_1) \arr ( x_1, x_1 ) \in \s S \\
\iff{} & \evmeet_{x\in \s A} x_1 \arr x_1 \in S 
\end{align*}
which is ensured by $\bf{i}\in S$.
\end{itemize}
        
        Therefore, $\pi_1$ induces a geometric inclusion $\Phi_{1} : P_{\A} \inc P_{\s\A}$.
\end{proof}

\begin{corollary}
    $\AT \A$ is a subtopos of $\AT {\s\A}$.
\end{corollary}

\begin{question}
In the case of $\A = \Pow {\mathbb{P}}$ for a discrete and absolute PCA, Johnstone \cite[Lem. 3.1]{johnstone2013} showed that there is another inclusion $P_{\A} \inc P_{\s\A}$, induced by the projection $\pi_0 : \s \A \to \A$ and disjoint from $\Phi_1$. We have not been able to show that this holds in general for (binary implicative) arrow algebras, nor to find reasonable assumptions under which this may be the case. 
\end{question}

At least in the modifiable case, we can say more about the inclusion ${\Phi_{1} : P_{\A} \inc P_{\s\A}}$.

Specializing \autoref{open subtriposes} to the context of arrow algebras, recall that a subtripos of $P_{\A}$ is \emph{open} if it is induced by a nucleus $o$ on $\A$ of the shape:
\[ o (a) \coloneqq u \arr a\]
for some $u \in A$, in which case the \emph{closed} nucleus:
\[c (a) \coloneqq a \lor u \]
induces its complement in the lattice of subtriposes of $P_{\A}$ considered up to equivalence.

\begin{definition}
    Given a modifiable arrow algebra $\A$, we define its \emph{modification} as the arrow algebra $\A^m \coloneqq (\s \A)_c$, where $c$ is the nucleus on $\s \A$ defined by: 
    \[c(x) \coloneqq x \lor (\bot,\top) \]
    
    We denote with $M_{\A}$ the \emph{modified arrow tripos} $P_{\A^m}$, that is, the subtripos $P_{(\s\A)_c}$ of $P_{\s\A}$.
\end{definition}

\begin{proposition}
Let $\A$ be a modifiable arrow algebra. Then, the inclusion $\Phi_{1} : P_{\A} \inc P_{\s\A}$ is open, induced by the nucleus:
            \[o( x ) \coloneqq (\bot,\top) \arr x\]  
In particular, $M_{\A}$ is the closed complement of $P_{\A}$ in the lattice of subtriposes of $P_{\s\A}$ considered up to equivalence.
\end{proposition}
\begin{proof} By \autoref{rem:nucleus induced by surjection} and the discussion preceding it, we only have to show that $o \liso\delta \pi_1$.
\begin{itemize}
    \item[--] To show that $o \vdash_{\s A}  \delta\pi_1$, note that:
    \begin{align*}
    \evmeet_{x\in \s A} o(x) \arr \delta\pi_1(x) \in \s S       \iff{} & \evmeet_{x\in \s A}( (\bot,\top) \arr (x_0,x_1) ) \arr (x_1,x_1) \in \s S\\ 
        \iff{}& \evmeet_{x\in \s A} \textstyle (\bot \arr x_0\evmeet \top \arr x_1, \top \arr x_1) \arr (x_1,x_1 ) \in \s S\\
    \iff{}& \evmeet_{x\in \s A}\textstyle (\bot \arr x_0\evmeet \top \arr x_1) \arr x_1 \evmeet (\top \arr x_1)\arr x_1 \in S\\
    \iff{}& \evmeet_{x\in \s A} (\top \arr x_1) \arr x_1 \in S
    \end{align*}
    which is ensured by the properties of $\partial a \coloneqq \top \arr a$.

\item[--] To show that $\delta\pi_1 \vdash_{\s A} o$ note that, by the hypothesis of modifiability:
    \begin{align*}
     \evmeet_{x\in\s A} \delta\pi_1(x) \arr o(x) \in \s S        \iff{} & \evmeet_{x\in \s A} (x_1,x_1) \arr  (\bot,\top) \arr (x_0,x_1)  \in \s S\\ 
        \iff{}& \evmeet_{x \in \s A}\textstyle (x_1,x_1) \arr (\bot\arr x_0 \evmeet \top \arr x_1, \top \arr x_1) \in \s S\\
            \iff{}& \evmeet_{x \in \s A}\textstyle (x_1,x_1) \arr ( \top \arr x_1, \top \arr x_1) \in \s S\\
        \iff{}& \evmeet_{x\in \s A} (x_1 \arr \top\arr x_1,x_1 \arr \top\arr x_1) \in \s S\\
        \iff{} & \evmeet_{x\in \s A} x_1 \arr \top\arr x_1 \in S
    \end{align*}
which is again ensured by the properties of $\partial a \coloneqq \top \arr a$.
\end{itemize}
\end{proof}

\begin{example}
    For $\A = \Pow{ \mathcal{K}_1}$, we reobtain what proved in \cite{vanoosten1997}: the effective topos $\Eff \iso \AT { \A }$ is an open subtopos of the effective topos built on the topos of sheaves over the Sierpiński space, $\Eff_{\cdot \arr\cdot} \iso \AT {\s\A}$, and Grayson's modified realizability topos $\Mod$ -- characterized in \cite{berg2023arrow} as $ \AT {\A^m}$ -- is its closed complement.
\end{example}

\begin{example}
    For $\A = \operatorname{PER}\NN$, we obtain that the \emph{extensional modified realizability topos} characterized in \cite{berg2023arrow} as $\AT { \A^m }$ is the closed complement of $\AT { \A}$ as subtoposes of $\AT {\s\A}$. 
\end{example}

Let us now see how the construction of the modified arrow tripos can be made pseudofunctorial. In the proof, we will need the following property, which makes use of the hypothesis of modifiability.

\begin{lemma}
    Let $f : \A \to \B$ be an implicative morphism in $\ArrAlgmod$. 
    
    Then, $c \s f c \vdash \s f c$.\footnote{Of course, the first $c$ is a nucleus on $\s\A$, while the second one is a nucleus on $\s \B$.}
\end{lemma}
\begin{proof}
By definition of the nucleus $c \in \N {\s\B}$, and using the fact that logical joins are computed pointwise in $(\s \B)^{\s A}$, $c \s f c \vdash \s f c$ is equivalent to:
    \[\evmeet_{x \in \s A} (\bot,\top) \arr \s f c (x) \in \s S_B\]
    Since $\B$ is modifiable, this reduces to:
\[ \evmeet_{x\in \s A} \top \arr Mf ( (c x)_1 ) \in S_B\]
where $M : \ArrAlg \to \ArrAlg$ is the monotonization pseudofunctor of \autoref{rem:monotonization}, and hence since $Mf \liso f$:
\[ \evmeet_{x\in \s A} \top \arr f ( ((\bot,\top) \lor (x_0,x_1) )_1 ) \in S_B\]

Note now that, in any arrow algebra of the form $\s \A$, the logical join $a \lor a'$ has $a_1 \lor a_1'$ as its second component. This can be seen using the explicit description of logical joins given in \cite[Prop. 5.1]{berg2023arrow} together with \autoref{rem:strong arrow algebras}, recalling also that (evidential) meets and implications in $\s\A$ are computed pointwise on the second component. Hence, the previous explicitly reads as:
\[ \evmeet_{x\in \s A} \top \arr f ( \top \lor x_1 ) \in S_B\]
which follows from $\evmeet_{a\in A} \top \arr (\top \lor a) \in S_A$ by the properties of implicative morphisms.
\end{proof}

\begin{theorem}
For every implicative morphism $f : \A \to \B$ in $\ArrAlgmod$, let $f^m : \A^m\to \B^m$ be the composite:
\[\begin{tikzcd}
    \A^m \rar["c"] & \s \A \rar ["\s f"] & \s \B \rar["\id_{\s B}"] & \B^m  
\end{tikzcd}\]
where $\s f : \s \A \to \s \B$ is the implicative morphism defined in \autoref{prop:sierpinski}.

Then, $(-)^m$ is a pseudofunctor $\ArrAlgmod\to \ArrAlg$.

Moreover, if $f$ is computationally dense with right adjoint $h : \B\to\A$, then $f^m$ is computationally dense as well, and a right adjoint is given by $h^m : \B^m \to \A^m$. Furthermore, the square:
\[\begin{tikzcd}
	{\B^m} & {\A^m} \\
	{\s \B} & {\s \A}
	\arrow["{{h^m}}", from=1-1, to=1-2]
	\arrow["c"', from=1-1, to=2-1]
	\arrow["c", from=1-2, to=2-2]
	\arrow["{{\s h}}"', from=2-1, to=2-2]
\end{tikzcd}\]
commutes up to isomorphism.
\end{theorem}
\begin{proof}
    First, let us show that $(-)^m$ preserves identities and compositions up to isomorphism.
\begin{itemize*}
    \item By definition, $\id_A^m : \A^m \to \A^m$ is given by the composite:
\[\begin{tikzcd}
    \A^m \rar["c"] & \s \A \rar ["\s {\id_A}"] & \s \A \rar["\id_{\s A}"] & \A^m  
\end{tikzcd}\]
which means that $\id_A^m = c : \A^m \to \A^m$, and obviously $c \liso^c \id_{\s A}$.

\item By definition, for $f : \A \to \B$ and $g : \B \to\mathcal{C}$, $(g f)^m$ is given by the composite:
\[\begin{tikzcd}
    \A^m \rar["c"] & \s \A \rar ["\s {(g f)}"] & \s {\mathcal{C}} \rar["\id_{\s C}"] & {\mathcal{C}^m}  
\end{tikzcd}\]
where of course $\s{(gf)}\liso \s g \s f$, whereas $g^m f^m$ is given by the composite:
\[\begin{tikzcd}[cramped]
     \A^m \rar["c"] & \s \A \rar ["\s f"] & \s \B \rar["\id_{\s B}"] & \B^m \rar["c"] & \s \B \rar ["\s {g}"] & \s {\mathcal{C}} \rar["\id_{\s C}"] & {\mathcal{C}^m}  
\end{tikzcd}\]
which means that we need to show that $\s g c \s f c \liso^c \s g \s f c$.

On one hand, using the fact that $\id \vdash c$ both for $c \in \N\B$ and $c \in \N{\mathcal{C}}$:
\[\s g \s f c \vdash \s g c \s f c \vdash c \s g c \s f c\]
i.e. $\s g \s f c \vdash^c \s g c \s f c$.

On the other, by the previous lemma we know that $c \s f c \vdash \s f c$, which implies $\s g c \s f c \vdash \s g \s f c$ by the properties of $\s g$, and hence $\s g c \s f c \vdash c \s g \s f c$ since $\id_{\s C} \vdash c$.
\end{itemize*}
The pseudofunctoriality of $f \mapsto \s f$ then yields the pseudofunctoriality of $f \mapsto f^m$.

Suppose now $h : \B\to\A$ is right adjoint to $f$; let us show that $\s h c$ is right adjoint to $\s f c :\A^m\to\B^m$.
\begin{itemize*}
    \item On one hand, $\s f c \s h c \vdash_{\s B}^c \id_{\s B}$ explicitly reads as $\s f c \s h c \vdash_{\s B} c$. By \autoref{prop:sierpinski}, we know that $\s h$ is right adjoint to $\s f$, so the previous is equivalent to $c \s h c \vdash_{\s B} \s h c$, which is ensured by the previous lemma.

    \item On the other, $\id_{\s A} \vdash^c_{\s A} \s h c \s f c$ explicitly reads as $\id_{\s A} \vdash_{\s A} c \s h c \s f c$. As $\id_{\s A} \vdash c$, this is ensured if $\id_{\s A} \vdash_{\s A} \s h c \s f c$. This is again equivalent to $\s f \vdash_{\s A} c \s f c$ as $\s h$ is right adjoint to $\s f$, which follows since $\id \vdash c$ both for $c \in \N{\s \A}$ and $c\in \N{\s \B}$.
    \end{itemize*}

Finally, to show that the square above commutes up to isomorphism, we need to show that $c h^m \liso \s h c$ as morphisms $\B^m \to \s \A$. On one hand, $\s h c \vdash c h^m$ explicitly means $\s h c \vdash_{\s B} c \s h c$, which follows simply from $\id_{\s A} \vdash c$. On the other, $c h^m \vdash \s h c $ explicitly means $c \s h c \vdash_{\s B} \s h c$, which is again ensured by the previous lemma.
\end{proof}

\begin{corollary}
    Let $\A$ and $\B$ be modifiable arrow algebras. Then, every geometric morphism $\Phi : P_{\B}\to P_{\A}$ induces a geometric morphism $\Phi^m : M_{\B} \to M_{\A}$ such that the diagram:
  \[\begin{tikzcd}
	{M_{\B}} & {M_{\A}} \\
	{P_{\s\B}} & {P_{\s\A}}
	\arrow["{{\Phi^m}}", from=1-1, to=1-2]
	\arrow[hook, from=1-1, to=2-1]
	\arrow[hook, from=1-2, to=2-2]
	\arrow["{\s \Phi}"', from=2-1, to=2-2]
\end{tikzcd}\]
is a pullback square of triposes and geometric morphisms.

In particular, the induced diagram of toposes and geometric morphisms:
\[\begin{tikzcd}
	{\AT {\B^m}} & {\AT {\A^m}} \\
	{\AT {\s \B}} & {\AT {\s\A}}
	\arrow[from=1-1, to=1-2]
	\arrow[hook, from=1-1, to=2-1]
	\arrow[hook, from=1-2, to=2-2]
	\arrow[from=2-1, to=2-2]
\end{tikzcd}\]
is a pullback square.
\begin{proof}
The fact that the square commutes follows directly by the previous proposition. To show that it is a pullback, instead, recall from \cite{johnstone2017} that, given a closed nucleus $k x \coloneqq x + u$ on $\s\A$, the pullback of the closed subtripos $P_{\s \A_k} \inc P_{\s\A}$ along $\s\Phi$ is the closed subtripos of $P_{\s\B}$ determined by the nucleus $k' y \coloneqq y + (\s \Phi)^+_{\s B} (u)$. Therefore, the square above is a pullback if and only if $(\s \Phi)^+_{\s B} (\bot,\top) \liso (\bot,\top)$ in $\s \B$. 

To prove this, let $f : \A \to \B$ be an implicative morphism with right adjoint $h : \B \to \A$ inducing $\Phi$, so that $\s \Phi$ is induced by $\s f$ with right adjoint $\s h$ as in \autoref{prop:sierpinski}; then, we need to show that $\s f (\bot,\top) \liso (\bot,\top)$ in $\s \B$. On one hand, by modifiability of $\B$, $(\bot,\top) \vdash \s f (\bot,\top)$ reduces simply to $\top \vdash f(\top)$, which is true as $f$ is an implicative morphism. On the other, $\s f (\bot,\top) \vdash (\bot,\top)$ is equivalent to $(\bot,\top) \vdash \s h (\bot,\top)$, which is true again as $\A$ is modifiable and $h$ is an implicative morphism.
\end{proof}
\end{corollary}

\begin{remark}
    In particular, restricting to arrow algebras of the form $\Pow {\mathbb{P}}$ for a discrete and absolute PCA $\mathbb{P}$, we reobtain \cite[Prop. 2.1]{johnstone2017}.
\end{remark}

\begin{remark}
  Recall by \autoref{prop:alternative subtripos} that we can identify $M_{\A}$ up to equivalence with the subtripos $M_{\A}' \inc P_{\s\A}$  defined by:
\begin{align*}
      M_{\A}'(I) \coloneqq{} & \Set{ \alpha \in P_{\s\A}(I) | c \alpha \vdash_I \alpha }\\
      ={} & \Set{ \alpha \in P_{\s\A}(I)| \evmeet_i (\bot,\top) \to \alpha(i) \in \s S_A}\\
  ={} & \Set{ \alpha \in P_{\s\A}(I)| \evmeet_i \top \to \alpha_1(i) \in  S_A}\\
  ={} & \Set{ \alpha \in P_{\s\A}(I)| \top_I \vdash_I \alpha_1 }
\end{align*}
and in the same way we can identify $M_{\B}$ up to equivalence with the subtripos $M_{\B}' \inc P_{\s\B}$ defined by:
  \[  M_{\B}'(I) = \Set{ \beta \in P_{\s\B}(I) | \top_I  \vdash_I \beta_1 }  \]

  In these terms, $\Phi^m$ can be described explicitly as:
  \[\begin{tikzcd}
{M'_{\B}} & {M'_{\A}}
\arrow[""{name=0, anchor=center, inner sep=0}, "{\s f {\circ} {-} }"', curve={height=12pt}, from=1-2, to=1-1]
	\arrow[""{name=1, anchor=center, inner sep=0}, "{\s h {\circ} {-} }"', curve={height=12pt}, from=1-1, to=1-2]
	\arrow["\dashv"{anchor=center, rotate=-90}, draw=none, from=0, to=1]
\end{tikzcd} \]
that is, exactly the restriction of $\s \Phi$ in both directions. 

\end{remark}

The details of the proof of the previous corollary also reveal that a similar result holds for open complements of modified triposes, again generalizing what proved in \cite{johnstone2017}.

\begin{proposition}
    Let $\A$ and $\B$ be modifiable arrow algebras. Then, for every geometric morphism $\Phi : P_{\B} \to P_{\A}$, the diagram:
\[\begin{tikzcd}
	{P_{\B}} & {P_{\A}} \\
	{P_{\s\B}} & {P_{\s\A}}
	\arrow["\Phi", from=1-1, to=1-2]
	\arrow[hook, from=1-1, to=2-1]
	\arrow[hook, from=1-2, to=2-2]
	\arrow["{\s \Phi}"', from=2-1, to=2-2]
\end{tikzcd}\]
    is a pullback square of triposes and geometric morphisms.

    In particular, the induced diagram of toposes and geometric morphisms:
    \[\begin{tikzcd}
	{\AT \B} & {\AT \A} \\
	{\AT {\s\B}} & {\AT {\s\A}}
	\arrow[from=1-1, to=1-2]
	\arrow[hook, from=1-1, to=2-1]
	\arrow[hook, from=1-2, to=2-2]
	\arrow[from=2-1, to=2-2]
\end{tikzcd}\]
is a pullback square.
\end{proposition}

\bibliographystyle{amsalpha}
\bibliography{mybib}{}

\providecommand{\bysame}{\leavevmode\hbox to3em{\hrulefill}\thinspace}
\providecommand{\MR}{\relax\ifhmode\unskip\space\fi MR }
\providecommand{\MRhref}[2]{%
  \href{http://www.ams.org/mathscinet-getitem?mr=#1}{#2}
}
\providecommand{\href}[2]{#2}
\begin{thebibliography}{vdBB23}

\bibitem[Bor94]{Borceux_1994}
Francis Borceux, \emph{Handbook of categorical algebra}, Cambridge University Press, 1994.

\bibitem[Bv02]{Birkedal2002-BIRRAM-2}
Lars Birkedal and Jaap {van Oosten}, \emph{Relative and modified relative realizability}, Annals of Pure and Applied Logic \textbf{118} (2002), no.~1, 115--132.

\bibitem[CMT21]{evframes2021}
Liron Cohen, Étienne Miquey, and Ross Tate, \emph{Evidenced frames: A unifying framework broadening realizability models}, Proceedings of the 36th Annual ACM/IEEE Symposium on Logic in Computer Science, 2021, pp.~1--13.

\bibitem[Fre13]{frey2013}
Jonas Frey, \emph{A fibrational study of realizability toposes}, Ph.D. thesis, Université Paris Diderot, 2013.

\bibitem[Fre24]{frey2024}
\bysame, \emph{Uniform preorders and partial combinatory algebras}, arXiv:2403.17340, 2024.

\bibitem[Gra81]{graymod}
Robin Grayson, \emph{Modified realizability toposes}, handwritten notes from Münster University, 1981.

\bibitem[Hig73]{higgs73}
Denis Higgs, \emph{A category approach to boolean-valued set theory}, 1973.

\bibitem[HJP80]{HJP80}
John M.~E. Hyland, Peter~T. Johnstone, and Andrew~M. Pitts, \emph{Tripos theory}, Mathematical Proceedings of the Cambridge Philosophical Society \textbf{88} (1980), no.~2, 205--232.

\bibitem[Hof06]{hofstra2006}
Pieter J.~W. Hofstra, \emph{All realizability is relative}, Mathematical Proceedings of the Cambridge Philosophical Society \textbf{141} (2006), 239 -- 264.

\bibitem[Hv03]{hvO03}
Pieter J.~W. Hofstra and Jaap {van Oosten}, \emph{Ordered partial combinatory algebras}, Mathematical Proceedings of the Cambridge Philosophical Society \textbf{134} (2003).

\bibitem[Joh02]{elephant}
Peter~T. Johnstone, \emph{Sketches of an elephant: a topos theory compendium}, Oxford logic guides, Oxford University Press, 2002.

\bibitem[Joh13a]{johnstone2013b}
\bysame, \emph{Geometric morphisms of realizability toposes}, Theory and Applications of Categories \textbf{28} (2013), no.~9, 241--249.

\bibitem[Joh13b]{johnstone2013}
\bysame, \emph{The gleason cover of a realizability topos}, Theory and Applications of Categories \textbf{28} (2013), no.~32, 1139--1152.

\bibitem[Joh17]{johnstone2017}
\bysame, \emph{Functoriality of modified realizability}, Tbilisi Mathematical Journal \textbf{10} (2017), 209--222.

\bibitem[JY20]{johnson20202dimensional}
Niles Johnson and Donald Yau, \emph{2-dimensional categories}, Oxford University Press, 2020.

\bibitem[Kre59]{Kreisel1959-KREIOA}
Georg Kreisel, \emph{Interpretation of analysis by means of constructive functionals of finite types}, Constructivity in mathematics (A.~Heyting, ed.), North-Holland Pub. Co., 1959, pp.~101--128.

\bibitem[Lon94]{longley95}
John Longley, \emph{Realizability toposes and language semantics}, Ph.D. thesis, University of Edinburgh, 1994.

\bibitem[Miq20a]{miquel2020-1}
Alexandre Miquel, \emph{Implicative algebras: a new foundation for realizability and forcing}, Mathematical Structures in Computer Science \textbf{30} (2020), no.~5, 458–510.

\bibitem[Miq20b]{miquel2020-2}
\bysame, \emph{Implicative algebras ii: completeness w.r.t. set-based triposes}, arXiv:2011.09085, 2020.

\bibitem[MM92]{Lane1992SheavesIG}
Saunders {Mac Lane} and Ieke Moerdijk, \emph{Sheaves in geometry and logic}, Universitext, Springer, 1992.

\bibitem[Pit81]{pitts1981}
Andrew~M. Pitts, \emph{The theory of triposes}, Ph.D. thesis, University of Cambridge, 1981.

\bibitem[SM17]{ferrer2019category}
Walter~Ferrer Santos and Octavio Malherbe, \emph{The category of implicative algebras and realizability}, Mathematical Structures in Computer Science \textbf{29} (2017), 1575 -- 1606.

\bibitem[vdBB23]{berg2023arrow}
Benno van~den Berg and Marcus Briet, \emph{Arrow algebras}, arXiv:2308.14096, 2023.

\bibitem[vO97]{vanoosten1997}
Jaap van Oosten, \emph{The modified realizability topos}, Journal of Pure and Applied Algebra \textbf{116} (1997), no.~1, 273--289.

\bibitem[vO08]{van2008realizability}
\bysame, \emph{Realizability: an introduction to its categorical side}, Elsevier Science, 2008.

\bibitem[Zoe22]{zoethout2022}
Jetze Zoethout, \emph{Computability models and realizability toposes}, Ph.D. thesis, Utrecht University, 2022.

\end{thebibliography}

\Acknowledgements

\end{document}